\documentclass[final]{siamltex}
\usepackage{epsfig,amssymb,amsmath,version}
\usepackage{amssymb,version,graphicx,fancybox,mathrsfs,multirow}

\usepackage{url,hyperref}
\usepackage[notcite,notref]{showkeys}
\usepackage{algorithm,algorithmic}
\usepackage{subfig}
\usepackage{color}
\usepackage{cases}
\usepackage{mathtools}
\allowdisplaybreaks

\newtheorem{remark}{Remark}[section]
\newtheorem{example}{Example}[section]

\definecolor{darkred}{rgb}{0.85,0,0}
\newcommand{\ZZ}[1]{{\color{black} #1}}

\def\al{\alpha}
\def\d{\mbox{d}}

\def\i{\mathrm{i}}
\def \Dal{\partial_t^\alpha}
\def\bU{{\bf{U}}}
\def\hbU{ {\bf{\widehat U}}}

\def\d{{\rm d}}

\input{siam.sty}

\title{A parallel-in-time algorithm for high-order BDF methods for diffusion and subdiffusion equations}

\author{Shuonan Wu\thanks{School of Mathematical Sciences, Peking University,  Beijing 100871, China.
Email address: snwu@math.pku.edu.cn}
\and Zhi Zhou\thanks{Department of Applied Mathematics, The Hong Kong Polytechnic University, Hung Hom, Hong Kong. Email address: zhizhou@polyu.edu.hk}}
\date{\today}

\begin{document}

\maketitle

\setlength\abovedisplayskip{5pt}
\setlength\belowdisplayskip{5pt}

\begin{abstract}
  In this paper,  we propose a parallel-in-time algorithm for
  approximately solving parabolic equations.  In particular, we
  apply the $k$-step backward differentiation formula, and
  then develop an iterative solver by using the waveform relaxation
  technique.  Each resulting iteration represents a periodic-like
  system, which could be further solved in parallel by using the
  diagonalization technique.  The convergence of the waveform relaxation iteration is
  theoretically examined by using the generating function method.  The
  approach we established in this paper extends the existing argument
  of single-step methods in Gander and Wu [Numer. Math., 143 (2019),
  pp.  489--527] to general BDF methods up to order six.  The argument
  could be further applied to the time-fractional subdiffusion
  equation, whose discretization shares common properties of the
  standard BDF methods, because of the nonlocality of the fractional
  differential operator.
   Illustrative numerical results are presented to complement the
  theoretical analysis.
\end{abstract}



\pagestyle{myheadings}
\thispagestyle{plain}

\section{Introduction}\label{sec:intro}

The aim of this paper is to develop a parallel-in-time (PinT) solver for
high-order time stepping schemes of  diffusion models.  We
begin with the normal diffusion, which is described by
parabolic equations.

Let $V\subset H=H'\subset V'$ be a Gelfand triple of complex Hilbert
spaces. Namely, the embedding $V\hookrightarrow H$ is continuous and
dense, and
$$
\langle u,v\rangle=(u,v),\quad\forall\, u\in H\hookrightarrow V',\,\,\forall\, v\in V\hookrightarrow H,
$$
where $\langle \cdot,\cdot\rangle$ is the duality pairing between $V'$
and $V$, and $(\cdot,\cdot)$ is the inner product on $H$.  Throughout,
we let $\|  \cdot \|$ and $\| \cdot \|_V$ denote the norms of the
space $H$ and $V$, respectively.

Let $T>0$, $v\in H$, $f\in L^2(0,T; V')$, and consider the initial
value problem of seeking $u\in L^2(0,T;V)\cap
H^1(0,T;V')\hookrightarrow C([0,T];H)$ such that
\begin{equation} \label{eqn:heat}
\left\{
\begin{aligned}
\partial_t u(t) + Au(t) &= f(t), \quad \text{for all} ~~t\in(0,T],  \\
u(0) &= v,
\end{aligned}
\right.
\end{equation}
where $A: V\to V'$ is a self-adjoint bounded linear operator  with the following elliptic property:
\begin{equation}\label{eqn:coe}
 \beta_0 \|u\|_{V}^2\le \langle A u, u\rangle \le \beta_1 \|u\|_{V}^2, \forall\,u\in V,
\end{equation}
with constants $\beta_1>\beta_0>0$.
\ZZ{For example, if we consider a heat equation on a bounded Lipschitz domain $\Omega\subset \mathbb{R}^d$
and $A$ denotes negative Laplacian $-\Delta$ with homogeneous
Dirichlet boundary conditions,
then $H=L^2(\Omega)$ and $V=H_0^1(\Omega)$.}

In recent years, the development and analysis of parallel algorithms
for solving evolution problems have attracted a lot of attention. The
first group of parallel schemes is based on the inverse Laplace
transform, which represents the solution as a contour integral
in the complex plane and a carefully designed quadrature rule
\cite{Sheen:2013, McLeanSloanThomee:2006,  SheenSloanThomee:2000,
Weideman:2007}.  {Such a method} is directly parallelizable and
accurate even for a nonsmooth problem data.
However, this strategy is not directly applicable for the nonlinear problem or anomalous diffusion
problems with time-dependent diffusion coefficients.
To the second group belongs the {widely used} parareal
algorithm \cite{Bal:2005, GanderJiang:SISC2013, GanderVandewalle:2007,
MadayTurinici:2002,  WuZhou:2015}, which could be derived as a
multigrid-in-time method or a multiple shooting method in the time
direction.  See also the space-time multigrid method
\cite{Hackbusch:book, Horton:1995, LubichOstermann:1987,
Weinzierl:2012, Vandewalle:1992}.  We refer the interested reader to
survey papers \cite{Gander:2015, BenjaminSchroder:2020} and references
therein.

Very recently, in \cite{GanderWu:2019}, Gander and Wu developed a
novel PinT algorithm, by applying the waveform relaxation
\cite{Miekkala:1987, Nevanlinna:1989b} and a diagonalization
technique \cite{Maday:2008,Gander:2020ParaDiag,GoddardWathen:2019}. In particular, they proposed a simple
iteration: for given $u_{m-1}(T) \in H$, look for $u_m$ such that
\begin{equation} \label{eqn:heat-iter-0}
\left\{
\begin{aligned}
\partial_t u_m(t) + Au_m(t) &= f(t), \quad \text{for all}~~ t\in(0,T];  \\
u_m(0) &= v + \kappa (u_m  -   u_{m-1} )(T) .
\end{aligned}\right.
\end{equation}
Here $\kappa$ denotes a relaxation parameter.
Note that the exact solution $u$ is a fixed point of the iteration \eqref{eqn:heat-iter-0}.
It was proved in \cite[Theorem 3.1]{GanderWu:2019} that, by selecting a proper $\kappa\in(0,1)$, the iteration \eqref{eqn:heat-iter-0}
converges with the convergence factor $  {\kappa e^{-cT}}/{(1-\kappa e^{-cT})}$, \ZZ{where $c$ is the smallest eigenvalue of $A$ and $c \geq \beta_0$.}

Then, a direct discretization of \eqref{eqn:heat-iter-0} by the backward Euler method immediately results
in a periodic-like discrete system, and therefore the diagonalization
technique is applicable here to carry out a direct
parallel computation.
\ZZ{The diagonalization technique was firstly proposed by Maday and R\o nquist for solving
evolution models \cite{Maday:2008}. The basic idea is to reformulate the time-stepping system into a space-time
all-at-once system, then diagonalize the time stepping matrix and solve all time steps in parallel.
The computational cost of each iteration is
proved to be $O([MN \log(N)+\widetilde{M}_f N]/p)$ for each
processor, where $M, N$ are numbers of degree
of freedoms in space and time respectively, $\widetilde{M}_f$ is the computational cost for solving a Poisson-like problem obtained by diagonalization,
and $p$ is the number of  used processors.}
In particular, Gander and Wu considered
single-step $\theta$-methods for solving \eqref{eqn:heat-iter-0} with
uniform step size \cite[Section 3.2]{GanderWu:2019}.  The convergence
analysis of the discrete system was also established, where the proof
is similar to the argument for the continuous problem.

With $\varepsilon$
being the desired error tolerance, a $k$-th order time stepping scheme
requires \ZZ{$N=O(\varepsilon^{-1/k})$} time steps. Therefore,
the   computational complexity for each
processor turns out to be $O([(\widetilde{M}_f + M |\log(\varepsilon)|/k)\varepsilon^{-1/k} ]/p)$.
Besides, the roundoff error of the algorithm is proved to be $O(\epsilon \kappa^{-2} N^{2}) = O(\epsilon \kappa^{-2}  \varepsilon^{-2/k})$,
where $\epsilon$ is the machine precision; see more details in Section \ref{ssec:dev-heat}.
Those facts motivate us to develop and analyze PinT schemes for \eqref{eqn:heat}, by using
some high-order time stepping schemes,  {such as the} $k$-step backward
differentiation formula (BDF$k$) with $k=2,3,\ldots,6$, and the
aforementioned waveform relaxation technique.  This is beyond the
scope of all existing references \cite{Gander:SISC-wave,
GanderWu:2019} which only focus on $\theta$-methods, and represents
the main theoretical achievements of the work.

Instead of directly discretizing \eqref{eqn:heat-iter-0}, we start
with the time stepping schemes of \eqref{eqn:heat} using BDF$k$
 with uniform step size $\tau=T/N$. Then, by perturbing the
discrete problem, we obtain a periodic-like system in each iteration, which can be
parallelly solved by using $O([MN \log(N)+\widetilde{M}_f N]/p)$ operations (for each
processor).  We prove that the resulting iteration linearly converges
to the exact solution (with a proper choice of the relaxation
parameter $\kappa$), by using the generating function technique as
well as the decay property of discrete solution operator.
Specifically, let $U_m^n$ be the solution of the $m$-th iteration of
the perturbed iterative system  with the initial
guess $U_0^n = v$ for all $0 \le n\le N$, and $u$ be the exact
solution to \eqref{eqn:heat}.  Provided certain data regularity, we show   the error estimate  for all $1 \le n\le N$
(Theorem \ref{thm:conv}):
\begin{equation*}
 \begin{aligned}
   \|U_m^n - u(t_n)\| &\le c(\gamma(\kappa)^m + \tau^k t_n^{-k}),\quad \text{with}~~ \kappa\in(0,1)~ \text{independent of}~ \tau,
 \end{aligned}
\end{equation*}
where the positive constant $c$ and the convergence factor \eqref{eq:gamma-heat}
$$ \gamma(\kappa) =  \frac{c\kappa e^{-c_1 T}}{1-c \kappa e^{-c_1 T}}  \in (0,1)$$
might depend on $k$, $\kappa$, $\beta_0$, $T$, $v$ and $f$,
but it is independent of $\tau$, $n$, $m$ and $u$.
\ZZ{Therefore,  to attain the discretization error $O(\tau^k)$ or $O(N^{-k})$, the computational
complexity for each processor is $O(\log(N)[MN \log(N)+\widetilde{M}_f N]/p)$.} 

The above argument could be
extended to the subdiffusion model, that involves a time-fractional
derivative of order $\alpha\in(0,1)$.  Let $T>0$, $v\in H$, $f\in
L^p(0,T; V')$ with $p\in(2/\alpha,\infty)$, and consider the initial
value problem of seeking $u\in L^p(0,T;V)\cap
H^\alpha(0,T;V')$ 
 such that
\begin{equation}\label{eqn:fde}
\left\{
\begin{aligned}
    \Dal u(t)  +  Au(t) &= f(t), \quad \text{for all}~~ t\in(0,T],\\
    u(0)&=v.
\end{aligned}
\right.
\end{equation}
Here $\partial_t^\alpha u$ denotes the left-sided Caputo fractional derivative of order $\alpha$, defined by 
\begin{equation}\label{McT}
   \partial_t^\alpha u(t):= \frac{1}{\Gamma(1-\al)}  \int_0^t(t-s)^{-\al}u'(s)\, ds,
\end{equation}

The interest in \eqref{eqn:fde} is motivated by its excellent
capability of modeling anomalously slow diffusion,  such as protein
diffusion within cells \cite{golding2006physical}, thermal diffusion
in media with fractal geometry \cite{Nigmatulin:1986}, and contaminant
transport in groundwater \cite{kirchner2000fractal}, to name but a
few.
The literature on the numerical approximation for the subdiffusion
equation \eqref{eqn:fde} is vast.  The most popular methods include
convolution quadrature
\cite{CuestaLubichPalencia:2006,JinLiZhou:SISC2017,BanjaiLopez:2019,Fischer:2019},
collocation method \cite{ZhuXu:2019, Stynes:2017, Liao:2018,
Kopteva:2019}, discontinuous Galerkin method
\cite{McLeanMustapha:2009, MustaphaAbdallahFurati:2014,
McLeanMustapha:2015}, and spectral method \cite{Chen:2020, HouXu:2017,
ZayernouriKarniadakis:2013}.  See also \cite{Baffet-1,
JiangZhang:2017, MR2385897, GasparRodrigo:2017, XuHesthavenChen:2015}
for some fast algorithms.

Our argument for linear multistep schemes could be easily applied to
many popular time stepping schemes for the subdiffusion problem
\eqref{eqn:fde}.  As an example, we consider the convolution
quadrature generated by BDF$k$ method, which was established by
Lubich's series of work \cite{Lubich:1986, Lubich:1988}.
Note that the fractional derivative is nonlocal-in-time, and
hence its discretization inherits the nonlocality and behaves like
a multistep discretization with an infinitely-wide stencil.  By
perturbing the time stepping scheme, we develop an iterative algorithm that requires a periodic-like system to be solved in each iteration,
which could be parallelly solved by a diagonalization technique with
\ZZ{$O([MN \log(N) + \widetilde{M}_f N]/p)$ operations for each processor}.  Moreover, error estimates of the
resulting numerical schemes are established by using the decay
property of the (discrete) solution operator.  We prove the following
error estimate (Theorem \ref{thm:conv-frac}):
\begin{equation*}
 \begin{aligned}
   \|U_m^N - u(t_N)\| &\le c(\gamma(\kappa)^m + \tau^k t_n^{-k}),\qquad \text{with}\quad \kappa=O(1/\log (N)),
 \end{aligned}
\end{equation*}
where $U_m^n$ is the solution to the iterative algorithm
\eqref{eqn:BDF-frac-iter} with the initial guess $U_0^n = v$ for all
$0 \le n\le N$, and $u$ is the exact solution to the subdiffusion
problem  \eqref{eqn:fde}.  In the estimate, the positive constant
$c$ and the convergence factor
\begin{equation*}
 \gamma(\kappa) =  \frac{c\kappa \log(N)}{1-c\kappa \log(N)}  \in (0,1)
\end{equation*}
\ZZ{might depend on $\alpha$, $k$, $\kappa$, $\beta_0$, $T$, $v$ and $f$}, but they
are always independent of $\tau$, $n$, $m$ and $u$.
The analysis is promising for some nonlinear
evolution equations as well as (sub-)diffusion problems involving
time-dependent diffusion coefficients. \ZZ{See a very recent work of Gu and Wu
 \cite{GuWu:JCP2020} for a parallel algorithm by using the diagonalization technique,
where the analysis only works for BDF$2$ with $\alpha<5/8$.}

The rest of the paper is organized as follows. In Section
\ref{sec:heat}, we develop a high-order time parallel schemes for
solving the parabolic problem and analyze its convergence by
generating function technique. In Section \ref{sec:frac}, we extend
our discussion to the nonlocal-in-time subdiffusion problem. Finally,
in Section \ref{sec:numerics}, we present some numerical results to
illustrate and complement the theoretical analysis.

\section{Parallel algorithms for normal diffusion equations}\label{sec:heat}
The aim of this section is to propose high-order multistep PinT schemes, with rigorous convergence analysis,
for approximately solving the parabolic equation \eqref{eqn:heat}.

\subsection{BDF$k$ scheme for normal diffusion equations}
We consider the BDF$k$ scheme, $k=1,2,\ldots,6$, with uniform step size.
Let $\{t_n = \tau n\}$ be a uniform partition of the interval $[0,T]$, with a time step size $\tau = T/N$.
For $n \ge 1$, the $k$-step BDF scheme seeks $U^n\in V$ such that  \cite{LiWangZhou:SINUM2020}
%
\begin{equation} \label{eqn:BDF}
\begin{aligned}
\bar\partial_\tau  U^{n} + AU^n &= f(t_n)
+ a_n^{(k)}  (f(0) -A v ) +
\sum_{\ell=1}^{k-2} b_{\ell,n}^{(k)}\tau^{\ell} \partial_t^{\ell}f(0)=:\bar f_n, \\
U^{-(k-1)}&=\cdots=U^{-1}=U^0=v.
\end{aligned}
\end{equation}
In \eqref{eqn:BDF} we use the BDF$k$ to approximate the first-order derivative by
$$
\bar{\partial}_t U^n := \frac{1}{\tau} \sum_{j=0}^{k} \omega_j U^{n-j}, 
$$
where constants $\{\omega_j\}$  are coefficients of the polynomials 
\begin{align}\label{eqn:bdf-gen}
 \delta_k (\zeta) := \sum_{\ell=1}^k \frac 1\ell  (1-\zeta)^\ell = \sum\limits^k_{j=0}\omega_j \zeta ^{j}.
\end{align}
For $\alpha=1$, the BDF$k$ scheme
is known to be $A(\vartheta_k)$-stable with angle $\vartheta_k= 90^\circ$, $90^\circ$, $86.03^\circ$,
$73.35^\circ$, $51.84^\circ$, $17.84^\circ$ for $k = 1,2,3,4,5,6$, respectively \cite[pp. 251]{HairerWanner:1996}.

\ZZ{To apply the BDF$k$ for parabolic problem, it is well-known that
one need starting data $U^j = u(t_j) + O(\tau^k)$ for $0\le j\le k-1$.
Then the  error bound of the time stepping scheme is $O(\tau^k)$.
However, for nonlocal-in-time subdiffusion models (which will be discussed in section \ref{sec:frac}),
the knowledge of $U^j$ for $0\le j\le k-1$ does not guarantee an error bound $O(\tau^k)$.
It is due to the lack of compatibility of problem data.
Fortunately, in the preceding work of the second author and his colleagues, it was proved that
one can recover the optimal error bound $O(\tau^k)$ by modifying the starting $k-1$ steps \cite{JinLiZhou:SISC2017}.
The strategy also works for the BDF$k$ for classical parabolic equations  \cite{LiWangZhou:SINUM2020}.
In order to keep consistency with numerical schemes for subdiffusion models,
we decide to apply the modified formulation \eqref{eqn:BDF} of the BDF$k$.}

\begin{table}[htb!]
\caption{The coefficients $a_n^{(k)}$ and $b_{\ell,n}^{(k)}$.}
\label{tab:anbn}
\centering
     \begin{tabular}{|c|ccccc|c|ccccc|}
\hline
      $k$   &  $a_1^{(k)}$  & $a_2^{(k)}$  & $a_3^{(k)}$ & $a_4^{(k)}$ & $a_5^{(k)}$
      &$\ell$ &$b_{\ell,1}^{(k)}$  & $b_{\ell,2}^{(k)}$  & $b_{\ell,3}^{(k)}$ & $b_{\ell,4}^{(k)}$ & $b_{\ell,5}^{(k)}$    \\[2pt]
\hline
       $k=2$       & $ \frac{1}{2}$ & & & & & & & &&&\\
\hline
       $k=3$          &$\frac{11}{12}$  & $-\frac{5}{12}$    &   &  &   &$\ell=1$   &$\frac{1}{12}$  &  0   &  &    & \\[2pt]
 \hline
       $k=4$          &$\frac{31}{24}$  & $-\frac{7}{6}$  & $\frac{3}{8}$ &  &           &$\ell=1$   &$\frac{1}{6}$  & $-\frac{1}{12}$    & $0$   &  &  \\[3pt]
             &  &     &  &     &                    &$\ell=2$   & $0$                &  $0$  & $0$   & & \\[2pt]
\hline
       $k=5$          &$\frac{1181}{720}$  & $-\frac{177}{80}$   & $\frac{341}{240}$ & $-\frac{251}{720}$  &      &$\ell=1$   &$\frac{59}{240}$  & $-\frac{29}{120}$   & $\frac{19}{240}$ & $0$ &  \\[3pt]
                &  &     &  &     &                    &$\ell=2$   & $\frac{1}{240}$                 & $-\frac{1}{240}$  & $0$ & $0$ &  \\[3pt]
                 &  &     &  &     &                   &$\ell=3$   &$\frac{1}{720}$ & $0$  & $0$ & $0$ & \\[2pt]
\hline
       $k=6$          &$\frac{2837}{1440}$& $-\frac{2543}{720}$   &$\frac{17}{5}$ & $-\frac{1201}{720}$ & $\frac{95}{288}$
        &$\ell=1$   &$\frac{77}{240}$& $-\frac{7}{15}$  &$\frac{73}{240}$ & $-\frac{3}{40}$  & 0\\[3pt]
           &  &     &  &     &                    &$\ell=2$   & $\frac{1}{96}$             & $-\frac{1}{60}$  & $\frac{1}{160}$  & $0$ & 0 \\[3pt]
          &  &     &  &     &                     &$\ell=3$   & $-\frac{1}{360}$ &  $\frac{1}{720}$ & $0$ & $0$  & 0\\[3pt]
           &  &     &  &     &                    &$\ell=4$   & $0$ & $0$ & $0$ & $0$  & 0\\[2pt]
\hline
     \end{tabular}
\end{table}

The next lemma provides some properties of the generating function.
The proof has been provided in \cite[Theorem A.1]{JinLiZhou:SISC2017}, and hence omitted here.\vskip5pt

\begin{lemma}\label{lem:delta}
For any $\varepsilon$, there exists $\theta_\varepsilon\in (\pi/2,\pi)$ such that for any
$\theta\in (\pi/2,\theta_\varepsilon)$, there holds $\delta_k (e^{-z}) \in \Sigma_{\pi-\vartheta_k+\varepsilon} $ for any $z\in \Gamma_{\theta}^{\pi/\sin{\theta}}= \{ z=re^{\pm i \theta}, ~0 \le r\le \pi/\sin\theta \}$, where
$\Sigma_\psi := \{z \in \mathbb{C}:~|\arg(z)| \leq \psi\}$. Meanwhile,
there exist positive constants $c_0$ and $c_0'$ such that
 \begin{equation*}
  \begin{aligned}
 c_0 |z| \le  |\delta_k (e^{-z})| \le c_0' |z| \qquad \forall~z\in \Gamma_{\theta}^{\pi/\sin{\theta}}.
\end{aligned}
\end{equation*}
 \end{lemma}\vskip5pt

For $n\ge k$, we choose $a_n^{(k)}$ and $b_{\ell,n}^{(k)}$  to be zero.
Then $\bar\partial_\tau u(t_n)$ is the standard approximation of $\partial_tu(t_n)$ by BDF$k$. 
For $1\le n\le k-1$, these constants have been determined in \cite{JinLiZhou:SISC2017, LiWangZhou:SINUM2020}, cf. Table \ref{tab:anbn}.
In particular, if
\begin{equation}\label{eqn:data}
v \in H \quad  \text{and} \quad f\in W^{k,1}(0,T;H) \cap C^{k-1}([0,T];H),
\end{equation}
the numerical solution to \eqref{eqn:BDF} satisfies the following error estimate. We omit the proof here and
refer interested readers
to \cite[Theorem 1.1]{LiWangZhou:SINUM2020} and \cite[Theorem 2.1]{JinLiZhou:SISC2017} for error analysis
of  (non-selfadjoint) parabolic systems and fractional subdiffusion equations, respectively.

\begin{lemma}\label{lem:error-BDF}
Assume that the problem data \ZZ{$v$ and $f$ satisfy \eqref{eqn:data}}, then
\begin{equation}\label{eqn:err-BDF}
  \begin{aligned}
  \|U^n-u(t_n)\|  \leq & c\, \tau^k
  \bigg(t_n^{ -k }  \| v \|
  + \sum_{\ell=0}^{k -1 } t_n^{ \ell -k +1 }  \|\partial_t^{\ell}f(0)\| +\int_0^{t_n} \|\partial_s^{k}f(s)\|d s\bigg) ,
  \end{aligned}
\end{equation}
where the constant $c$ is independent of $\tau$ and $n$.
\end{lemma}


\subsection{Development of parallel-in-time algorithm}\label{ssec:scheme-heat}
Next, we develop a PinT algorithm for  \eqref{eqn:BDF}: for given $U_{m-1}^n$, $N-k+1\le n\le N$, we
compute $U_{m}^n$ by
\begin{equation} \label{eqn:BDF-iter}
\begin{aligned}
\bar\partial_\tau  U_m^{n} + AU_m^n &=  \bar f_n, ~\quad\qquad\qquad\qquad\qquad n = 1,2,\cdots, N,\\
U_m^{-j} &= v + \kappa(U_m^{N-j} - U_{m-1}^{N-j}),\quad j=0,1,\ldots,k-1,
\end{aligned}
\end{equation}
where the revised source term $\bar f_n$ is given in \eqref{eqn:BDF}.
Note that the exact time stepping solution $\{U^n\}_{n=1}^N$ is a
fixed point of this iteration.  In Section \ref{ssec:conv-heat}, we provide a
systematic framework to study the iterative algorithm
\eqref{eqn:BDF-iter}, which also works for the time-fractional
subdiffusion problem discussed later.

We may rewrite the BDF$k$ scheme \eqref{eqn:BDF-iter} in the
following matrix form:
\begin{equation} \label{eq:BDF-k-matrix}
\frac{1}{\tau}(B_k(\kappa) \otimes I_x){\bf U}_m + (I_t \otimes A)
 {\bf U}_m ={\bf F}_{m-1},
\end{equation}
where ${\bf U}_m = (U_m^1, U_m^2, \cdots, U_m^N)^T$, ${\bf
F}_{m-1}=(F_1,F_2,\cdots,F_N)^T$ with
\begin{equation} \label{eqn:BDF-k-F}
  F_n :=  \bar f_n + \frac{\kappa}{\tau}\sum_{j=1}^n\omega_j
  U_{m-1}^{N-j+1} + \frac{1}{\tau}\sum_{j=0}^{n-1} \omega_j v,
\end{equation}
and
$$
B_k(\kappa) =
\begin{bmatrix}
\omega_0 &\cdots &   0 &  \kappa \omega_k& \cdots & \kappa \omega_2& \kappa \omega_1\\
\omega_1 & \omega_0 & &     0 &    \kappa \omega_k& \cdots  &  \kappa \omega_2 \\
 \vdots&\vdots & &  &  & \ddots& \vdots  \\
 \vdots&\vdots & &   \ddots &   & 0& \kappa \omega_k\\
\omega_k& \omega_{k-1}& \cdots&  &       & & 0\\
 0 & \omega_{k} & \omega_{k-1} & \cdots  & &   &\vdots  \\
   &   &&      \ddots  & \ddots  &  & \vdots \\
& &  0 & \omega_k& \cdots&     \omega_1&\omega_0
\end{bmatrix}.
$$
Here, we recall that $\omega_j = 0$ if $j > k$ and $\sum_{j=0}^k
\omega_j = 0$ for the normal diffusion equations. The following lemma
is crucial for the design of PinT algorithm.

\begin{lemma}[Diagonalization]\label{lem:diag-heat}
Let $\Lambda(\kappa) = \mathrm{diag}(1, \kappa^{-\frac1N}, \cdots,
\kappa^{-\frac{N-1}{N}})$, then
$$
B_k(\kappa) = \Lambda(\kappa) \tilde{B}_k(\kappa)
  \Lambda(\kappa)^{-1},
$$
where the circular matrix $\tilde{B}_k(\kappa)$ has the form
$$
\tilde{B}_k(\kappa) =
\begin{bmatrix}
\omega_0 &     0 &  \kappa^{\frac{k}N} \omega_k& \cdots &  \kappa^{\frac2N} \omega_2& \kappa^{\frac1N} \omega_1\\
\kappa^{\frac1N}\omega_1 & \omega_0 &       0 &    \kappa^{\frac{k}N}  \omega_k& \cdots  & \kappa^{\frac2N}  \omega_2 \\
 \vdots&\vdots & &  &    \ddots& \vdots  \\
 \vdots&\vdots &     \ddots &    & 0&\kappa^{\frac {k }N} \omega_k\\
\kappa^{\frac kN} \omega_k& \kappa^{\frac {k-1}N} \omega_{k-1}& \cdots&  &        & 0\\
 0 & \kappa^{\frac kN} \omega_{k} & \kappa^{\frac {k-1}N} \omega_{k-1} & \cdots      &  \\
   &   &&      \ddots  & \ddots  &    \vdots \\
&   0 & \kappa^{\frac kN} \omega_k& \cdots&    \kappa^{\frac 1N} \omega_1&\omega_0
\end{bmatrix}.
$$
As a consequence, $B_k(\kappa)$ can be diagonalized by
$$
  B_k(\kappa) = S(\kappa) D_k(\kappa) S(\kappa)^{-1}, \quad
  S(\kappa):=\Lambda(\kappa) V,
$$
  where the Fourier matrix
 \begin{equation}  \label{eqn:Fourier-matrix}
  V = [v_1,v_2,\ldots,v_N], \quad \text{with}\quad v_n = \big[1,\mathrm{e}^{i\frac{2(n-1)\pi}{N}}, \ldots,
  \mathrm{e}^{i\frac{2(n-1)(N-1)\pi}{N}} \big]^T,
  \end{equation}
  and $D_k(k)$ is a diagonal matrix.
\end{lemma}

Using the above lemma, we can solve the system \eqref{eqn:BDF-iter} in a parallel-in-time manner.
\begin{algorithm}[hbt!]
  \caption{~~PinT algorithm by diagonalization technique for diffusion equation.\label{alg:heat}}
  \begin{algorithmic}[1]
    \STATE  Solve $(S(\kappa)\otimes I_x){\bf H} = {\bf F}_{m-1}$.
    \STATE  Solve $({D}_k(\kappa)\otimes I_x + \tau
  I_t\otimes A) {\bf Q} = \tau {\bf H}$.
    \STATE Solve $(S(\kappa)^{-1}\otimes I_x) {\bf U}_m =
   {\bf Q}$.
  \end{algorithmic}
\end{algorithm}

It is known that the circulant matrix can be diagonalized by FFT with
$\mathcal{O}(N\log N)$ operations \cite[Chapter 4.7.7]{Golub:book}.
Then in each iteration, it turns
out to be $N$ independent Poisson-like equations, which can be
efficiently solved by, for instance, the multigrid method.\vskip5pt

\paragraph{Speedup analysis of Algorithm \ref{alg:heat}} Let
$M_f$ be the total real floating point operations for solving the (elliptic)
Poisson-like equations. Then the
cost for the serial computation is $O(M_f N)$.
\ZZ{It is known that with
(optimal) multigrid method $M_f$ is proportional to $M$, the number of degrees of freedom in space.}

Consider the parallelization of Algorithm \ref{alg:heat} with $p$
processors. The {\bf Step 1} and {\bf Step 3} can be finished
with the total computational cost $O([MN \log(N)] / p)$ by using the bulk
synchronous parallel FFT algorithm \cite{IndaBisseling:2001}, see also
\cite[Section 4.1]{GanderWu:2019} for the detailed analysis. We note
here that the ${\bf F}$ in \eqref{eqn:BDF-k-F} needs to be updated via
${\bf U}_{m-1}$, whose computational cost is $O(M)$ since $\omega_j =
0$ when $j > k$. The PinT algorithm of ${\bf F}$ will be
more delicate for the time-fractional subdiffusion problem discussed
later.

For the {\bf Step 2}, the total computational cost is $O((\widetilde{M}_fN)/p)$,
where $\widetilde{M}_f$ denotes total real floating point operations for solving the
Poisson-like equations obtained by diagonalization,
and hence the cost for the parallel computation is $O([MN\log(N) + \widetilde{M}_f N]/p)$.
In case that $p = O(N)$, then the parallel computational cost reduces to $O(M\log(N) + \widetilde{M}_f)$.
\ZZ{In some cases, $\widetilde{M}_f$ could be (almost) linear to $M$ even in higher dimensions.
For example, if we consider the heat equation with periodic boundary conditions, we can apply FFT to solve
the Poinsson-like equations with computational complexity $\widetilde{M}_f = O(M\log M)$.}

\ZZ{With $\varepsilon$
being the desired error tolerance, a $k$-th order time stepping scheme
requires $N=O(\varepsilon^{-1/k})$ time steps. Then, in order to attain sufficient accuracy $O(\varepsilon)$,
the number of iterations should be $O(\log N)$,
because the convergence factor $\gamma(\kappa)\in (0,1)$ is
independent of $N$ (cf. Theorem 2.7).
Therefore, the total computational cost is $O([MN(\log N)^2 + \widetilde{M}_f N \log N]/p)$. }\vskip5pt

\ZZ{
\begin{remark}\label{FFT}
The proposed algorithm uses FFT to diagonalize over time.
It transforms the parabolic PDE into a decoupled set of elliptic PDEs.
Such a technique is quite common for solving time-periodic problems,
but relatively new for solving initial value problems.

In this paper, we only discuss the parallelism in the time direction. Nevertheless,
combining the proposed method with some parallel-in-space algorithms,
one may obtain an algorithm with polylog parallel complexity if $O(MN)$ processors were available. For example,
for parabolic equations with periodic boundary conditions, one may apply FFT
to transform the parabolic PDE into a decoupled set of ODEs:
$$ u_j '(t) + \lambda_j u_j(t)  = f_j(t),\quad \text{with}~~ u_j(0)=v_j, $$
where $1\le j\le M$ and $M$ denotes the number of degree
of freedoms in space. Then, these decoupled ODEs could be efficiently solved by using
 the parallel-in-time algorithm proposed in this paper.
\end{remark}}

\paragraph{Roundoff error of Algorithm \ref{alg:heat}}
Let $\bU_m$ be the exact solution of \eqref{eq:BDF-k-matrix},
and $\hbU_m$ be the solution of Algorithm \ref{alg:heat}. We assume that
the {\bf Step 2} of Algorithm \ref{alg:heat} is solved in a direct
manner (for example, the LU factorization).  Then for simplicity, we
consider an arbitrary eigenvalue of the matrix $A$ and analyze
the relative roundoff error. To this end, we replace matrices $I_x$ and $A$
by scalars $1$ and $\mu$. Then the system \eqref{eq:BDF-k-matrix} reduces to
$$B {\bU_m} = F, \quad \text{with}\quad B = \frac{1}{\tau} B_k(\kappa) + \mu I_t. $$
Then by Lemma \ref{lem:diag-heat}, we define
$$ D = S(\kappa)^{-1} B S(\kappa),\quad \text{and hence}\quad  D = \frac{1}{\tau} D_k(\kappa) + \mu I_t .$$
Note that to solve $B {\bU_m} = F$ by diagonalization is equivalent to solving $(B+\delta B) {\hbU_m} = F$
with some perturbation $\delta B$, which can be easily bounded by \cite[p. 496]{GanderWu:2019}
$$ \| \delta B  \|_2 \le \epsilon (2N+1) \| S(\kappa) \|_2 \| S(\kappa)^{-1} \|_2 \| D \|_2 + O(\epsilon^2), $$
where $\epsilon$ denotes the machine precision ($\epsilon= 2.2204\times10^{-16}$ for a 32-bit computer).
Then the roundoff error satisfies
\begin{equation*}
\begin{split}
\frac{\| {\bU_m} - {\hbU_m}  \|_2}{\| {\bU_m} \|_2} &\le \text{cond}_2 ( B) \frac{\| \delta  B \|_2}{\|   B \|_2}
\le \epsilon(2N+1) \| S(\kappa) \|_2 \| S(\kappa)^{-1} \|_2 \| D \|_2 \|  B^{-1} \|_2 \\
&\le \epsilon(2N+1)\, \text{cond}_2(S(\kappa))^2\, \text{cond}_2( D) .
\end{split}
\end{equation*}
Here we note that for $\kappa\in(0,1]$
\begin{equation*}
\begin{split}
\| S(\kappa) \|_2 \le \| \Lambda(\kappa ) \|_2 \| V \|_2 \le \kappa^{-\frac{N-1}{N}} \sqrt{N},
\end{split}
\end{equation*}
and
\begin{equation*}
\begin{split}
\| S(\kappa)^{-1} \|_2 \le \| \Lambda(\kappa )^{-1} \|_2 \| V ^{-1} \|_2 \le \frac{1}{\sqrt{N}}.
\end{split}
\end{equation*}
Therefore, we arrive at
\begin{equation*}
\begin{split}
 \text{cond}_2 (  S(\kappa)) = \| S(\kappa) \|_2 \| S(\kappa)^{-1} \|_2  \le \kappa^{-\frac{N-1}{N}} \le \kappa^{-1}.
\end{split}
\end{equation*}
For the diagonal matrix $  D$
\begin{equation*}
\begin{split}
\|  D^{-1}\|_2  &= \max_{1\le n\le N}  \Big|\frac{1}{\tau}\delta_k(\kappa^{\frac1N}e^{-\i\frac{2(n-1)\pi}{N}}) + \mu\Big|^{-1}
\le (\sin\theta_k\mu)^{-1}.
\end{split}
\end{equation*}
\ZZ{Here we apply $A(\theta_k)$ stability of the BDF$k$ scheme with $\theta_k\in(0,\pi/2)$. Similarly,
using the definition of generating function \eqref{eqn:bdf-gen}, we derive for any $\kappa\in(0,1)$
and $N>1$
\begin{equation}\label{D2}
\begin{split}
\|  D \|_2  &= \max_{1\le n\le N} \Big|\frac{1}{\tau}\delta_k(\kappa^{\frac1N}e^{-\i\frac{2(n-1)\pi}{N}})  + \mu\Big|
\le   \frac{1}{\tau}\delta_k(-\kappa^{\frac1N})  + \mu  \le \frac{1}{\tau}\delta_k(-\kappa^{\frac1N})  + \mu.
\end{split}
\end{equation}
Hence, we obtain
\begin{equation*}
\begin{split}
\text{cond}_2(   D) &\le \Big(\frac{\delta_k (-1)}{\tau} +\mu \Big)(\mu\sin\theta_k)^{-1} .
\end{split}
\end{equation*}
As a result,  we have for $\mu\in [\mu_0,\infty)$ (where the positive number $\mu_0$ depends on $\beta_0$ in \eqref{eqn:coe})
\begin{equation}\label{eqn:rounderr-heat}
\frac{\| {\bU_m } - {\hbU_m}  \|_2}{\| {\bU_m} \|_2}
\le  \epsilon (2N+1) \kappa^{-2} \Big(\frac{\delta_k (-1)}{\tau} +\mu \Big)(\mu\sin\theta_k)^{-1} \le C_k \epsilon  \kappa^{-2} N^{ 2},
\end{equation}
where the constant $C_k$ can be written as
$$C_k = 3\Big(1+\Big[\frac{\delta_k (-1)}{T}\Big] /\mu_0\Big)(\sin\theta_k)^{-1}.$$
It only depends on the order of BDF method and $\beta_0$ in \eqref{eqn:coe}.}
Note that the bound of roundoff error is uniform for $\mu\rightarrow\infty$,
therefore it holds for all self-adjoint operators $A$ satisfying \eqref{eqn:coe}.

\ZZ{\begin{remark}\label{rem:N2}
The above analysis shows a reasonable estimate that roundoff error is $O(\epsilon\kappa^{-2} N^2)$.
See a similar estimates for some A-stable single-step methods like backward
Euler scheme or Crank--Nicolson scheme \cite{GanderWu:2019}. The difference
is that the BDF$k$ scheme (with $k>2$) is no longer A-stable.
In our numerical experiments, we indeed observe that the roundoff error increases as $N\rightarrow\infty$.
\end{remark} }

\subsection{Representation of numerical solution}\label{ssec:op-heat}
The aim of this section is to develop the representation of the
numerical solution of the $k$-step BDF schemes through a contour
integral in complex domain, and to establish decaying properties of
the solution operators.

By letting $W^n = U^n-v$, we can reformulate the time stepping scheme
\eqref{eqn:BDF} as
\begin{equation}\label{eqn:w-disc}
\begin{aligned}
\bar\partial_\tau  W^{n} + AW^n &= -Av + \bar f_n , \quad n = 1,2,\cdots, N,\\
W^{-(k-1)}&=\cdots=W^{-1}=W^0=0.
\end{aligned}
\end{equation}
By multiplying $\xi^n$ on \eqref{eqn:w-disc} and taking summation over $n$ (\ZZ{we extend $n$ in \eqref{eqn:w-disc} to infinity in sense that $\bar f_n=0$ with $n\ge N$}), we have
\begin{equation*}
  \sum_{n=1}^{\infty} \xi^n \bar \partial_t W^n + \sum_{n=1}^{\infty}
  \xi^n AW^n=- \sum_{n=1}^{\infty} \xi^n Av + \bar f_n \xi^n.
\end{equation*}
For any given sequence $\{V^n\}_{n=0}^\infty$, let
$\widetilde{V}(\xi):=\sum_{n=0}^{\infty}V^n\xi^n$ denote its
generating function.  Since $W^{-(k-1)}=\cdots=W^{-1}=W^0=0$,
according to properties of discrete convolution, we have the identity
$$
\sum_{n=1}^{\infty} \xi^n\bar\partial_\tau
V_n=\frac{\delta_k(\xi)}{\tau}\widetilde{V}(\xi),
$$
where $\delta_k(\xi)$ denotes the generating function of the $k$-step
BDF method \eqref{eqn:bdf-gen}. Therefore we conclude that
\begin{equation*}
 \Big (\frac{\delta_k(\xi)}{\tau} + A\Big)\widetilde{W}(\xi)=
  -\Big(\frac{\xi} {1-\xi}\Big) Av + \widetilde{\bar f_n}(\xi),
\end{equation*}
which implies that
\begin{equation*}
 \widetilde{W}(\xi) = \Big(\frac{\delta_k(\xi)}{\tau} + A\Big)^{-1}\Big[-\Big(\frac{\xi}
  {1-\xi}\Big) Av + \widetilde{\bar f_n}(\xi)\Big].
\end{equation*}
It is easy to see that
$\widetilde{W}(\xi)$ is analytic with respect to $\xi$ in the circle $|\xi|=\rho$, for $\rho>0$ small,
on the complex plane, then with Cauchy's integral formula, we have the following expression
\begin{equation*}
\begin{aligned}
  W^n&=\frac{1}{2\pi\i}\int_{|\xi|=\rho} \xi^{-n-1}\widetilde{W}(\xi) \d\xi\\
  &=\frac{\tau}{2\pi\i}\int_{|\xi|=\rho} \xi^{-n-1} \Big( \delta_k(\xi)  + \tau A\Big)^{-1}\Big[-\Big(\frac{\xi}
  {1-\xi}\Big) Av + \widetilde{\bar f_n}(\xi)\Big] \d\xi.
\end{aligned}
\end{equation*}
%
%
%
Therefore we obtain the solution representation:
 \begin{equation}\label{eqn:solrep-wn-2}
\begin{aligned}
 U^n = (I+ F_\tau^n) v + \tau \sum_{j=1}^n E_\tau^{n-j}\bar f_j.
\end{aligned}
\end{equation}
where the discrete operators $F_\tau^n$ and $E_\tau^n$ are respectively defined by
\ZZ{\begin{equation}\label{eqn:disc-sol-op}
\begin{aligned}
F_\tau^n &= -\frac{1}{2\pi\i}\int_{|\xi|=\rho}   \frac{1}
  {\xi^n(1-\xi)} \Big( \delta_k(\xi)/\tau  +  A\Big)^{-1} A   \, \d\xi,\\
E_\tau^n &=\frac{1}{2\pi\tau\i}\int_{|\xi|=\rho}  \xi^{-n-1}
  \Big( \delta_k(\xi) /\tau + A\Big)^{-1}   \, \d\xi.
\end{aligned}
\end{equation}}

Now we recall a useful estimate (cf. \cite[Lemma 10.3]{Thomee:2006}).
For $k=1,\cdots,6$, there are positive constants $c, C$ and $\lambda_0$ \ZZ{(only depends on the BDF$k$ method)}
such that
\begin{equation}\label{eqn:est-gen}
\Big| \frac{1}{2\pi\i}\int_{|\xi|=\rho}  \xi^{-n-1}
  ( \delta_k(\xi)  +  \lambda )^{-1}   \, \d\xi  \Big|\le ~\left\{
\begin{aligned}
   & C e^{-cn\lambda},\quad  0<\lambda\le \lambda_0,\\
   & C \lambda^{-1} e^{-cn},\quad \lambda>\lambda_0.
\end{aligned}
\right.
\end{equation}

This together with the coercivity property \eqref{eqn:coe} immediately implies the following lemma.

\begin{lemma}\label{lem:sol-op}
Let $E_\tau^n$ be the discrete operator  defined in \eqref{eqn:disc-sol-op}. Then
\begin{equation*}
 \| E_\tau^{n} \|_{H \rightarrow H} \le c_2 e^{-c_1 t_n}.
\end{equation*}
Here the generic positive constants $c_1$ and $c_2$ are independent of $n$ and $\tau$.
\end{lemma}
%

\subsection{Convergence analysis}\label{ssec:conv-heat}
In this section, we analyze the convergence of the iterative scheme \eqref{eqn:BDF-iter},
or equivalently,
\begin{equation} \label{eqn:BDF-iter-1}
\begin{aligned}
\bar\partial_\tau  U_m^{n} + AU_m^n  &= \bar f_n - \frac{\kappa}\tau G^n_m, \quad n = 1,2,\ldots, N,\\
U_m^{-j} &= v  ,\qquad\qquad\quad j=0,1,\ldots,k-1.
\end{aligned}
\end{equation}
where the term $G^n_m$ is given by
\begin{equation}\label{eqn:Gn}
G^n_m =   \sum_{j=n}^{k} \omega_j \left(U_m^{N+n-j} - U_{m-1}^{N+n-j}\right).
\end{equation}
Here, the summation is assumed to vanish if the lower bound is greater than the upper bound.
We aim to show that $U_m^N$ 
 converges to $U^N$, the solution of the time stepping scheme \eqref{eqn:BDF}, as $m\rightarrow\infty$.

\begin{lemma}\label{lem:conv-heat}
 Let $U_m^n$ be the solution to the iterative scheme \eqref{eqn:BDF-iter} with $v = 0$ and $\bar f_n=0$ for all $n=1,2,\ldots.N$.
 Then we can choose a proper parameter $\kappa>0$ in \eqref{eqn:BDF-iter}, which is independent of step size $\tau$, such that the following estimate holds valid:
\begin{equation*}
   \sum_{j=0}^{k-1} \|U_m^{N-j}\| \le \gamma(\kappa) \sum_{j=0}^{k-1} \| U_{m-1}^{N-j}\|.
\end{equation*}
 Here $\gamma(\kappa)\in(0,1)$, is constant depending on $\kappa$, $\beta_0$ and $T$, but independent of $\tau$.
\end{lemma}
\begin{proof}
Following the preceding argument in Section \ref{ssec:op-heat}, $U_m^n$ could be represented by
\begin{equation}\label{eqn:heat-umn}
\begin{aligned}
U_m^n &=- \kappa \sum_{i=1}^{n} E_\tau^{n-i}  G_m^{i} \\
& =- \kappa \sum_{i=1}^{\min\{k,n\}} E_\tau^{n-i}
\sum_{j=i}^{k} \omega_j (U_m^{N+i-j} -    U_{m-1}^{N+i-j}), \quad  n = 1, 2, \ldots, N.
\end{aligned}
\end{equation}
Now we take the $H$ norm in \eqref{eqn:heat-umn} and apply Lemma \ref{lem:sol-op} to obtain
\begin{equation*}
\begin{aligned}
\sum_{j=0}^{k-1} \|U_m^{N-j}\| &\le c\kappa e^{-c_1 T}  \sum_{j=0}^{k-1} \| U_m^{N-j} -  U_{m-1}^{N-j}\|.
\end{aligned}
\end{equation*}
where $c$ is a generic constant and $c_1$ is the constant in Lemma \ref{lem:sol-op}.
Then we apply the triangle inequality, by choosing $\kappa$ small enough such that $c\kappa e^{-c_1 T}<1$, and hence derive that
\begin{equation*}
\begin{aligned}
\sum_{j=0}^{k-1} \|U_m^{N-j}\| &\le \frac{c\kappa e^{-c_1 T}}{1-c\kappa e^{-c_1 T}} \sum_{j=0}^{k-1} \| U_{m-1}^{N-j}\|.
\end{aligned}
\end{equation*}
Finally, we define the convergence factor
\begin{equation}\label{eq:gamma-heat}
\gamma(\kappa) :=  \frac{c\kappa e^{-c_1 T}}{1-c\kappa e^{-c_1 T}}.
\end{equation}
Then by choose $\kappa$ sufficiently small such that $c\kappa e^{-c_1 T} \in (0,1/2)$, we have $\gamma(\kappa) \in (0,1)$. 
This completes the proof of the desired assertion.
\end{proof}

 \begin{corollary}\label{cor:conv-heat-1}
 Let $U_m^n$ be the solution to the iterative scheme \eqref{eqn:BDF-iter}, and $U^n$
 be the solution to the $k$-step BDF scheme  \eqref{eqn:BDF}. Then we can choose a proper parameter $\kappa>0$
 in \eqref{eqn:BDF-iter}, which is independent of step size $\tau$, such that the following estimate holds valid:
\begin{equation*}
 \sum_{j=0}^{k-1}  \|U^{N-j} - U_m^{N-j}\| \le \gamma(\kappa) \sum_{j=0}^{k-1} \| U^{N-j} - U_{m-1}^{N-j}\|,\quad \forall ~~m\ge1.
\end{equation*}
 Here $\gamma(\kappa)\in(0,1)$, is constant depending on $\kappa$, $\beta_0$ and $T$, but independent of $\tau$.
 \end{corollary}
 \begin{proof}
 To this end, we let $e_m^n = (U_m^n - U^n)$, and note that the time stepping solution
$\{U^n\}_{n=1}^N$ is the fixed point of the iteration \eqref{eqn:BDF-iter}. Therefore $e_m^n$ satisfies
\begin{equation*}
\begin{aligned}
\bar\partial_\tau  e_m^{n} + Ae_m^n  &= - \frac{\kappa}\tau K^n_m , \quad n = 1,2,\cdots, N,\\
e_m^{-j} &=  0 ,\qquad\qquad j=0,1,\ldots,k-1.
\end{aligned}
\end{equation*}
where the term $K_n^m$ is given by
\begin{equation*}
\begin{aligned}
K^n_m =
\sum_{j=n}^{k} \omega_j (e_m^{N+n-j} -    e_{m-1}^{N+n-j}  ).
\end{aligned}
\end{equation*}
Then the convergence estimate follows immediately from Theorem \ref{thm:conv}.
\end{proof}

Combining Corollary \ref{cor:conv-heat-1} with the estimate \eqref{eqn:err-BDF}, we have the following
error estimate of the iterative scheme \eqref{eqn:BDF-iter}.

\begin{theorem}\label{thm:conv}
Suppose that the assumptions \eqref{eqn:coe} and \eqref{eqn:data} are valid.
Let $U_m^n$ be the solution to the iterative scheme \eqref{eqn:BDF-iter}
with the initial guess $U_0^n = v$ for all $0 \le n\le N$, and $u$
be the exact solution to  the parabolic equation \eqref{eqn:heat}.
Then by choosing proper relaxation parameter $\kappa\in(0,1)$ which is  independent of step size $\tau$,
 the following estimate holds valid:
\begin{equation*}
 \begin{aligned}
   \|U_m^n - u(t_n)\| &\le c(\gamma(\kappa)^m + \tau^k t_n^{-k}), \quad \text{for all}~~ n=1,2,\ldots,N.
 \end{aligned}
\end{equation*}
Here constants $\gamma(\kappa)\in(0,1)$ and $c>0$  might depend on $k$, $\kappa$, $\beta_0$, $T$, $v$ and $f$,
but they are  independent of $m$, $n$, $\tau$ and $u$.
\end{theorem}
\begin{proof}
We split the error into two parts:
$$ U_m^n - u(t_n) = (U_m^n - U^n) + (U^n - u(t_n)), $$
where $U^n$
is the solution to the $k$-step BDF scheme  \eqref{eqn:BDF}.
Note that the second term has the error bound  \eqref{eqn:err-BDF}. Meanwhile,
via \eqref{cor:conv-heat-1},  we have the estimate
 \begin{equation*}
\begin{aligned}
  \sum_{j=0}^{k-1} \|U^{N-j} - U_m^{N-j}\| \le \gamma(\kappa)^m \sum_{j=0}^{k-1} \| U^{N-j} -v\|
  \le \gamma(\kappa)^m \| v\| + \gamma(\kappa)^m \sum_{j=0}^{k-1} \| U^{N-j}\|
\end{aligned}
\end{equation*}
By the error estimate \eqref{eqn:err-BDF} and the assumption of data regularity \eqref{eqn:data}, we obtain that
$$\sum_{j=0}^{k-1} \| U^{N-j}\| \le c_T.$$
This, \eqref{eqn:heat-umn} and Lemma \ref{lem:sol-op}  lead to the estimate that
 \begin{equation*}
\begin{aligned}
 \|  U^{n} - U_m^{n} \| \le c \gamma(\kappa) \Big(\sum_{j=0}^{k-1} \|U^{N-j} - U_m^{N-j}\| + \sum_{j=0}^{k-1} \|U^{N-j} - U_{m-1}^{N-j}\| \Big)
 \le c  \gamma(\kappa)^m.
\end{aligned}
\end{equation*}
Then we obtain the desired result.
\end{proof}

\ZZ{
\begin{remark}\label{rem:choice-kappa-1}
For backward Euler method (BDF$1$), the convergence rate was proved to be \cite{GanderWu:2019}
$$  \gamma(\kappa) = \frac{\kappa e^{-c_1 T}}{1- \kappa e^{-c_1 T}} < \frac{\kappa  }{1- \kappa },  $$
So the iterative algorithm converges linearly by choosing $\kappa<1/2$, and the smaller parameter $\kappa$ leads to the faster convergence. However, in Section \ref{ssec:scheme-heat}, we have shown that the roundoff error is proportional to $O(\kappa^{-2})$,
so a tiny $\kappa$ may lead to a disastrous roundoff error.
Therefore one needs to choose a proper $\kappa\in(0,1/2)$ in order to
balance the roundoff error and the convergence rate.

For $k$-step BDF methods with $1<k\le 6$, we obtain a similar result
$$  \gamma(\kappa) = \frac{ c    \kappa e^{-c_1 T}}{1- c    \kappa e^{-c_1 T}} \le  \frac{c  \kappa  }{1- c \kappa},  $$
with an extra factor $c>1$.
This is due to the different stability estimate in Lemma \ref{lem:sol-op} of linear multistep methods.
Even though it is hard to derive an explicit bound of the generic constant $c$ for $k$-step BDF methods,
our empirical experiments show that the choice $\kappa\approx 0.1$  leads to an acceptable roundoff error ($\approx 10^{-12}$),
and meanwhile the convergence is very fast (see Fig. \ref{fig:eg1-kappa}).
 Note that the convergence rate is independent of $N$, so the increase
in the total number of steps
will not affect the robust convergence.\end{remark}}

\section{Parallel algorithms for nonlocal-in-time subdiffusion equations}\label{sec:frac}
In the section, we shall consider the subdiffusion equations \eqref{eqn:fde}, which involve a fractional-in-time derivative of
order $\alpha\in(0,1)$. The fractional-order differential operator is nonlocal, and its discretization  inherits
this nonlocality and looks like a multistep discretization with an infinitely wide stencil.
This motivates us to extend the argument established in Section \ref{sec:heat} to the subdiffusion equations \eqref{eqn:fde}.

\subsection{BDF$k$ scheme for subdiffusion equations}
To begin with, we discuss the development of a PinT algorithm for \eqref{eqn:fde}.
We apply the convolution quadrature (CQ) to discretize the fractional derivative  on uniform grids.
Following the same setting in Section \ref{ssec:scheme-heat}, let $\{t_n=n\tau\}_{n=0}^N$ be a uniform partition of the time interval $[0,T]$,
with a time step size $\tau=T/N$.

CQ was first proposed by Lubich \cite{Lubich:1986,
Lubich:1988} for discretizing Volterra integral equations. This approach provides a systematic framework
to construct high-order numerical methods to discretize fractional derivatives, and has been the
foundation of many early works.
%
%
Specifically, CQ approximates
the Riemann-Liouville derivative $^R\partial_t^\alpha \varphi(t_n)$ \ZZ{with $\alpha\in(0,1)$}, which is defined by
\begin{equation*}
   ^R\partial_t^\alpha \varphi := \frac{\d}{\d t}\frac{1}{\Gamma(1-\alpha)}\int_0^t(t-s)^{-\alpha}\varphi(s)\d s,
\end{equation*}
(with $\varphi(0)=0$) by a discrete convolution  (with the shorthand notation $\varphi^n=\varphi(t_n)$)
\begin{equation}\label{eqn:CQ}
\hat\partial_\tau^\alpha \varphi^n:=\frac{1}{\tau^\alpha}\sum_{j=0}^n \omega_j^{(\alpha)} \varphi^{n-j}.
\end{equation}
Here we consider the BDF$k$ method for example, then
the weights $\{\omega_j^{(\alpha)}\}_{j=0}^\infty$ are the coefficients in the power series expansion
\begin{equation}\label{eqn:delta}
\delta_k (\xi)^\alpha= \sum_{j=0}^\infty \omega_j^{(\alpha)} \xi ^j,
\end{equation}
where $\delta_k(\xi)$ is given by \eqref{eqn:bdf-gen}.
Generally, the weights $\{\omega_j^{(\alpha)}\}_{j=0}^{\infty}$ can be computed either by the fast
Fourier transform or recursion \cite{Podlubny:1999}. 

The next lemma provides a useful bound of the coefficients $\omega_j^{(\alpha)}$.

\begin{lemma}\label{lem:coeff}
The weights $\omega_n^{(\alpha)}$ satisfy the estimate that $|\omega_n^{(\alpha)}| \le c (n+1)^{-\alpha-1}$, where the constant $c$ only depends on $\alpha$ and $k$.
\end{lemma}
\begin{proof}
The case of $k=1$ has been proved in \cite[Lemma 12]{JinZhou:iter}, by using the expression of the coefficients:
$\omega_n^{(\alpha)}=-\prod_{j=1}^n(1-\frac{1+\alpha}{j}).$
However, the closed forms of coefficients of high-order schemes are not available.
Here we provide a systematic proof for all BDF$k$ methods, $k=1,2,\ldots,6$.

By the definition of $\{\omega_j^{(\alpha)}\}$ and Cauchy's integral formula, we obtain that
\begin{equation*}
\begin{aligned}
 \omega_j^{(\alpha)} = \frac{1}{2\pi \i} \int_{|\xi|=1} \delta_k(\xi)^\alpha \xi^{-n-1} \,\d \xi
 = \frac{\tau^{1+\alpha}}{2\pi \i} \int_{  \Gamma_\tau} e^{z t_n }(\delta_k(e^{-z\tau})/\tau)^\alpha   \,\d z,
 \end{aligned}
\end{equation*}
where $\Gamma^\tau:=\{z=\i y: y\in\mathbb{R},\,|y|\leq \pi/\tau\}$.
 The analyticity together with the periodicity of the integrand
allows the deformation of the contour to
\begin{equation*}
  \Gamma_{\psi}^\tau =  \{z=r e^{\pm \i\psi}: 0\le r\le  \pi/\tau\sin \psi  \},
\end{equation*}
with $\psi \in (\pi/2,\pi)$. Then Lemma
\ref{lem:delta} implies for $n\ge1$
\begin{equation*}
\begin{aligned}
 |\omega_j^{(\alpha)}| & \le c \tau^{1+\alpha} \int_{\Gamma_{\psi }^\tau}  |e^{z t_n}| |\delta_k(e^{z\tau})|^\alpha   \,|\d z|
 \le c  \tau^{1+\alpha}  \int_0^{\frac{\pi}{\tau\sin\psi}} e^{ r t_n\cos\psi } r^\alpha \, \d \, r \\
 &\le c \tau^{1+\alpha} t_n^{-\alpha-1} \le c n^{-\alpha-1} \le 4 c (n+1)^{-\alpha-1}.
  \end{aligned}
\end{equation*}
This together with the uniform bound of $\omega_0^{(\alpha)}$ leads to the desired result.
\end{proof}\vskip5pt

Using the relation
$\partial_t^\alpha \varphi(t) = {^R\partial_t^\alpha}(\varphi-\varphi(0))$, see e.g.  \cite[p. 91]{KilbasSrivastavaTrujillo:2006},
the subdiffusion problem could be rewritten into the form
\begin{equation*}
  ^R\partial_t^\alpha(u-v) + A u = f.
\end{equation*}
Then the time stepping scheme based on the CQ for problem \eqref{eqn:fde} is to seek
approximations $U^n$ to the exact solution $u(t_n)$ by
\begin{equation}\label{eqn:BDF-CQ-0}
\hat \partial_\tau^\alpha  (U^n - v)  + A U^n = f(t_n),\quad n=1,\dots,N.
\end{equation}
By the definition of the discretized operator $\hat \partial_\tau^\alpha$ in \eqref{eqn:CQ}, we have
\begin{equation}\label{eqn:CQ-1}
\hat \partial_\tau^\alpha  (U^n - v) = \frac{1}{\tau^\alpha} \sum_{j=0}^n \omega_j^{(\alpha)}(U^{n-j} - v)
= \frac{1}{\tau^\alpha} \sum_{j=0}^\infty \omega_j^{(\alpha)} U^{n-j} =: \bar \partial_\tau^\alpha U^n,
\end{equation}
by setting the historical initial data
\begin{equation}\label{eqn:his-ini}
 U^{n} = v,\qquad \text{for all}\quad n\le 0.
\end{equation}
Then we reformulate the time stepping scheme \eqref{eqn:BDF-CQ-0}-\eqref{eqn:his-ini} by
\begin{equation} \label{eqn:BDF-CQ-1}
\begin{aligned}
\bar\partial_\tau ^\alpha U^{n} + AU^n &= f(t_n) \quad n = 1,2,\cdots, N,\\
U^{n}&=v,~\qquad n\le0.
\end{aligned}
\end{equation}
If the exact solution $u$ is smooth and has sufficiently many vanishing derivatives at $t=0$, then
the approximation $U^n$ converges at a rate of $O(\tau^k)$ uniformly in time $t$ \cite[Theorem 3.1]{Lubich:1988}. However, it
generally only exhibits a first-order accuracy when solving fractional evolution equations
 even for smooth $v$ and  $f$ \cite{CuestaLubichPalencia:2006,JinLiZhou:SISC2017}, because the
requisite compatibility conditions
$$ Av + f(0) = 0,\quad \text{and}\quad {\partial_t^\ell}f(0) = 0~~\forall \ell=1,2,\ldots,k,  $$
are usually not satisfied.
This loss of accuracy is one distinct feature for most time stepping schemes deriving
under the assumption that the solution $u$ is sufficiently smooth.

In order to restore the high-order convergence rate, we simply modify the starting steps
\cite{CuestaLubichPalencia:2006,JinLiZhou:SISC2017,LubichSloanThomee:1996,YanFord:2018}.
In particular, for $n \ge 1$, the CQ-BDF$k$ scheme seeks $U^n\in V$ such that
%
\begin{equation} \label{eqn:BDF-frac}
\begin{aligned}
\bar\partial_\tau ^\alpha U^{n} + AU^n &= f(t_n)
+ a_n^{(k)}  (f(0) -A v ) +
\sum_{\ell=1}^{k-2} b_{\ell,n}^{(k)}\tau^{\ell} \partial_t^{\ell}f(0)=:\bar f_n, \\
U^{n}&=v,\quad n\le0.
\end{aligned}
\end{equation}

The coefficients $a_n^{(k)}$ and $b_{\ell,n}^{(k)}$ have to be chosen appropriately (cf. Table \ref{tab:anbn}).
For $n\ge k$,  $a_n^{(k)}$ and $b_{\ell,n}^{(k)}$  are zero, then $\bar\partial_\tau^\alpha u(t_n)$ is the standard CQ-BDF$k$ scheme,
that approximates $\partial_t^\alpha u(t_n)$. 
Then there holds the following error estimate \cite[Theorem 2.1]{JinLiZhou:SISC2017}.

\ZZ{
\begin{lemma}\label{lema:frac-err}
If the initial data $v$ and forcing data $f$ satisfy
\begin{equation}\label{eqn:data-frac}
v\in H \quad  \text{and} \quad f \in W^{k,\frac1\alpha+\epsilon}(0,T;H) ~\text{with some}~\epsilon>0,
\end{equation}
then the time stepping solution $U^n$ to \eqref{eqn:BDF-frac} satisfies the following error estimate:
\begin{equation}\label{eqn:err-BDF-frac}
  \begin{aligned}
  \|U^n-u(t_n)\|  \leq & c\, \tau^k
  \bigg(t_n^{ -k }  \|v\|
  + \sum_{\ell=0}^{k -1 } t_n^{ \alpha+\ell -k }  \|\partial_t^{\ell}f(0)\| +\int_0^{t_n} (t-s)^{\alpha-1}\|\partial_s^{k}f(s)\|d s\bigg) ,
  \end{aligned}
\end{equation}
where the constant $c$ is independent of $\tau$ and $t_n$.
\end{lemma}}

\subsection{Development of parallel-in-time scheme}\label{ssec:dev-heat}
In order to develop a parallel solver for the time stepping method \eqref{eqn:BDF-frac}, we apply the strategy developed in Section \ref{sec:heat}.
For given $U_{m-1}^n$, $\le n\le N$, we compute $U_{m}^n$ by
\begin{equation} \label{eqn:BDF-frac-iter}
\begin{aligned}
\bar\partial_\tau^\alpha  U_m^{n} + AU_m^n &=  \bar f_n, ~\quad\qquad\qquad\qquad\qquad n = 1,2,\ldots, N,\\
U_m^{\ZZ{-n}} &= v + \kappa(U_m^{\ZZ{N-n}} - U_{m-1}^{\ZZ{N-n}}),\quad n=0, 1, \ldots, N-1,\\
U_m^{n} &= v ,\qquad\qquad\qquad\qquad\qquad n\le -N,
\end{aligned}
\end{equation}
where the revised source term $\bar f_n$ is given in \eqref{eqn:BDF-frac}.
Note that  $\{U^n\}_{n=1}^N$, the exact time stepping solution to \eqref{eqn:BDF-frac}, is a fixed point of this iteration.
We shall examine  convergence  in Section \ref{ssec:conv-frac}.

Now we may rewrite the perturbed BDF$k$ scheme \eqref{eqn:BDF-frac-iter} in the following matrix form:
\begin{equation} \label{eq:BDF-k-matrix-frac}
\frac{1}{\tau^\alpha}(B_k(\kappa) \otimes I_x){\bf U}_m + (I_t \otimes A) {\bf U}_m = {\bf F}_{m-1},
\end{equation}
where ${\bf U}_m = (U_m^1, U_m^2, \cdots, U_m^N)^T$, ${\bf F}_{m-1}=(F_1,F_2,\cdots,F_N)^T$ with
\begin{equation} \label{eq:BDF-subdiff-F}
F_n =  \bar f_n + \frac{\kappa}{\tau^\alpha}\sum_{j=n}^{N-1}\omega_j U_{m-1}^{N+n-j} + \frac1{\tau^\alpha}\sum_{j=0}^{n-1}\omega_jv,
\end{equation}
and
$$
B_k(\kappa) =
\begin{bmatrix}
\omega_0 &  \kappa \omega_{N-1}&   \cdots &    \kappa \omega_2& \kappa \omega_1\\
\omega_1 & \omega_0 &      \cdots  & \kappa\omega_3 &  \kappa \omega_2 \\
\omega_2&\omega_1 &     \ddots& &\vdots  \\
 \vdots&\vdots & &   \omega_0 &      \kappa \omega_{N-1}\\
\omega_{N-1}& \omega_{N-2}& \cdots&    \omega_1     &  \omega_0\\
\end{bmatrix}.
$$

Similar as Lemma \ref{lem:diag-heat}, we have the following result.
\begin{lemma}[Diagonalization]\label{lem:diag-frac}
Let $\Lambda(\kappa) = \mathrm{diag}(1, \kappa^{-\frac1N}, \cdots,
\kappa^{-\frac{N-1}{N}})$, then
$$
B_k(\kappa) = S(\kappa) D_k(\kappa)
  S(\kappa)^{-1}, \quad S(\kappa) = \Lambda(\kappa)V,
$$
where $V$ is the Fourier matrix defined in \eqref{eqn:Fourier-matrix}.
\end{lemma}

The above lemma implies the parallel solver for \eqref{eqn:BDF-frac-iter}.

\begin{algorithm}[hbt!]
  \caption{~~PinT BDF$k$ scheme for subdiffusion.\label{alg:frac}}
  \begin{algorithmic}[1]
    \STATE Solve $(S(\kappa)\otimes I_x){\bf H} = {\bf F}_{m-1}$.
    \STATE Solve $({D}_k(\kappa)\otimes I_x + \tau^\alpha
  I_t\otimes A) {\bf Q} = \tau^\alpha {\bf H}$.
    \STATE Solve $(S(\kappa)^{-1}\otimes I_x) {\bf U}_m = {\bf Q}$.
  \end{algorithmic}
\end{algorithm}

\paragraph{Speedup analysis of Algorithm \ref{alg:frac}}
Due to the nonlocality of the fractional-order differential operator,
the discretized operator \eqref{eqn:CQ-1} requires the information of
all the previous steps.  \ZZ{In particular,  in the $n$-th step of CQ-BDF$k$ scheme,
we need to solve a poisson-like problem:
\begin{equation*}
(\omega_0^{(\alpha)} I + \tau^\alpha A) U^n = \sum_{j=1}^n \omega_j^{(\alpha)} ( U^0 -U^{n-j}) + \omega_0^{(\alpha)}  U^0 + \bar f^n.
\end{equation*}
The computation cost of this step is $O(nM+M_f)$. Then taking summation over $n$ from $1$ to $N$, we derive that the total computational cost of
 the direct implementation of
CQ-BDF$k$ scheme is $O(M N^2 +M_fN)$.}

Consider the parallelization of Algorithm \ref{alg:frac} with $p$ used
processors. Similar to the \ZZ{discussion} on Algorithm \ref{alg:heat}, the
cost of parallel FFT in {\bf Step 1} and {\bf Step 3} is
$O([MN\log(N)]/p)$. We check the computation cost of ${\bf F}_{m-1}$ in {\bf
Step 1}, which contains the following three components:
\eqref{eq:BDF-subdiff-F}:
\begin{enumerate}
\item The source term $\bar{f}_n$ defined in \eqref{eqn:BDF-frac}: The
correction is taken at first few steps and hence the computation
cost is $O((MN)/p)$.
\item The convolution term $\sum_{j=n}^{N-1} \omega_j U_{m-1}^{N+n-j}$
can be rewritten as the $n$-th entry of
$$
\begin{bmatrix}
  \omega_N & \omega_{N-1} & \cdots & \omega_2 & \omega_1 \\
  0 & \omega_N & \cdots & \omega_3 & \omega_2 \\
  \vdots   & \vdots & \ddots & \vdots & \vdots \\
  0 & 0 & \cdots & \omega_N & \omega_{N-1} \\
  0 & 0 & \cdots & 0 & \omega_N
\end{bmatrix}
\begin{bmatrix}
  U_{m-1}^1 \\ U_{m-1}^2 \\ \vdots \\ U_{m-1}^{N-1} \\ U_{m-1}^N
\end{bmatrix}
:= {\bf W}{\bf U}_{m-1}.
$$
Although ${\bf W}$ is not circulant, the above matrix can be extended
to be circulant for the purpose of using FFT algorithm. More
precisely, consider
$$
\begin{bmatrix}
  {\bf Z} & {\bf W} \\
  {\bf W} & {\bf Z}
\end{bmatrix}
\begin{bmatrix}
{\bf 0} \\ {\bf U}_{m-1}
\end{bmatrix}, \quad \text{with}
\quad
{\bf Z}:=
\begin{bmatrix}
  0 & 0 & \cdots & 0 & 0 \\
  \omega_1 & 0 & \cdots & 0 & 0 \\
  \vdots   & \vdots & \ddots & \vdots & \vdots \\
  \omega_{N-2} & \omega_{N-3} & \cdots & 0 & 0 \\
  \omega_{N-1} & \omega_{N-2} & \cdots & \omega_1 & 0
\end{bmatrix}.
$$
It can be easily seen that the extended matrix is circulant.  Thanks
again to the bulk synchronous parallel FFT algorithm \cite{IndaBisseling:2001}, the FFT
of the extended system and $[{\bf 0}, {\bf U}_{m-1}]^T$ lead to the
computation costs $O([N\log(N)]/p)$ and $O([MN\log(N)]/p)$,
respectively. By using the inverse FFT, the computation cost of the
convolution term $\sum_{j=n}^{N-1} \omega_j U_{m-1}^{N+n-j}$ turns out
to be $O([MN\log(N)]/p)$.
\item The convolution term $\sum_{j=0}^{n-1}\omega_jv$ can be computed
  with the cost $O([N\log(N) + MN]/p)$.
\end{enumerate}
For the {\bf Step 2}, the total computational cost is $O([\widetilde M_fN]/p)$.
To sum up, the overall cost for the parallel computation is $O([MN\log(N) + \widetilde{M}_f N]/p)$,
which becomes $O( M  \log N   + \widetilde{M_f} )$ if $p=O(N)$.\vskip5pt

\ZZ{Similar to the discussion on Algorithm \ref{alg:heat}, in order to attain desired accuracy,
the total computational cost is $O([MN(\log N)^2 + \widetilde{M}_f N \log N]/p)$ for each processor. }\vskip5pt

\ZZ{
\paragraph{Roundoff error of Algorithm \ref{alg:frac}}
In case that $I_x = 1$ and $A = \mu$, by using the same argument in
Section \ref{ssec:scheme-heat},
the roundoff error can be bounded by
 \begin{equation*}
\frac{\| {\bU}_m - {\hbU}_m  \|_2}{\| {\bU}_m \|_2}
\le  \epsilon (2N+1) \kappa^{-2} \Big(\Big[\frac{\delta_k (-1)}{\tau}\Big]^\alpha +\mu \Big)(\mu\sin(\max(\alpha(\pi-\theta_k),\pi/2)))^{-1}.
\end{equation*}
For $\mu\in[\mu_0,\infty)$ and $\alpha\in(0,1)$, we have the uniform estimate
\begin{equation}\label{eqn:rounderr-frac}
\begin{split}
\frac{\| {\bU}_m - {\hbU}_m  \|_2}{\| {\bU}_m \|_2} &\le   \epsilon (2N+1) \kappa^{-2} \Big(1+ \Big[\frac{\delta_k (-1)}{\tau}\Big]^\alpha /\mu_0 \Big)(\sin(\max(\alpha(\pi-\theta_k),\pi/2)))^{-1}\\
&\le C_{k} \epsilon \kappa^{-2} N^{1+\alpha},
\end{split}
\end{equation}
where the constant $C_{k}$  is
$$C_k = 3\Big(1+\Big[\frac{\delta_k (-1)}{T}\Big]^\alpha /\mu_0\Big)(\sin(\max(\alpha(\pi-\theta_k),\pi/2)))^{-1}.$$
Note that the constant $C_k$ depends on the order of BDF method, $\beta_0$ in \eqref{eqn:coe}, the terminal time $T$ and the fractional order $\alpha$.}

\subsection{Representation of numerical solution and convergence analysis} \label{ssec:conv-frac}
Next, we represent the solution of the time stepping scheme
\eqref{eqn:BDF-frac}  as a contour integral in complex domain.
\ZZ{Following the argument in Section \ref{ssec:op-heat},
the solution to the time stepping scheme \eqref{eqn:BDF-frac} can be written as
 \begin{equation}\label{eqn:solrep-frac}
\begin{aligned}
 U^n = (I+F_\tau^n)   v + \tau \sum_{j=1}^n E_\tau^{n-j}\bar f_j.
\end{aligned}
\end{equation}}
where the discrete operators $F_\tau^n$ and $E_\tau^n$ are respectively defined by
\begin{equation}\label{eqn:disc-sol-op-frac}
\begin{aligned}
F_\tau^n &= -\frac{1}{2\pi\i}\int_{|\xi|=\rho}   \frac{1}
  {\xi^n(1-\xi)} \Big( (\delta_k(\xi)/\tau)^\alpha  + A\Big)^{-1} A   \, \d\xi,\\
E_\tau^n &=\frac{1}{2\pi \tau \i}\int_{|\xi|=\rho}  \xi^{-n-1}
  \Big( (\delta_k(\xi)/\tau)^\alpha  +  A\Big)^{-1}   \, \d\xi.
\end{aligned}
\end{equation}

The following lemma provides the decay properties of the discrete solution operator.
The proof is standard, see, for example, \cite{LubichSloanThomee:1996, JinLiZhou:SISC2017}, and hence it is omitted.
\begin{lemma}\label{lem:sol-op-frac}
For the solution operators $E_\tau^n$ defined by \eqref{eqn:disc-sol-op-frac}, there holds that
$$ \| E_\tau^n \|_{H\rightarrow H} \le c t_{n+1}^{\alpha-1},\quad \forall ~~n=1,2,.\ldots,N.  $$
where the constant $c$ is independent of $\tau$ and $n$.
\end{lemma}


In this section, we aim to show the convergence of the iterative method \eqref{eqn:BDF-frac-iter}
by choosing an appropriate parameter $\kappa$. 
Equivalently, the  scheme \eqref{eqn:BDF-iter} could be reformulated as
 \begin{equation} \label{eqn:BDF-frac-iter-1}
\begin{aligned}
\bar\partial_\tau^\alpha  U_m^{n} + AU_m^n  &= \bar f_n- \frac{\kappa}{\tau^\alpha} G^n_m, \quad n = 1,2,\cdots, N,\\
U_m^{n} &= v,\quad n\le 0.
\end{aligned}
\end{equation}
where the term $G^n_m$ is given by
\begin{equation}\label{eqn:Gn-frac}
G^n_m =   \sum_{j=0}^{N-1} \omega_{n+j}^{(\alpha)} \left(U_m^{N-j} - U_{m-1}^{N-j}\right).
\end{equation}
We aim to show that $U_m^N$ 
 converges to $U^N$, the solution of CQ-BDF$k$ scheme \eqref{eqn:BDF-frac}, as $m\rightarrow\infty$.

 \begin{lemma}\label{lem:conv-frac}
 Let $U_m^n$ be the solution to the iterative algorithm \eqref{eqn:BDF-frac-iter} with $v=0$ and $\bar f_n=0$ for all $n=1,2,\ldots,N$.
 Then we can choose a proper parameter $\kappa=O(1/\log(N))$
 in \eqref{eqn:BDF-iter}, such that the following estimate holds valid:
 \begin{equation*}
    \| U_m^{n} \| \le c t_n^{\alpha-1}\gamma(\kappa)^m \Big(\tau\sum_{j=1}^{N} t_{j}^{-\alpha} \| U_{0}^{N-j+1}\|\Big).
 \end{equation*}
 \ZZ{Here $\gamma(\kappa)\in(0,1)$ might depend on $\alpha$, $\kappa$, $\beta_0$ and $T$}, but independent of $\tau$, $n$ and $m$.
 \end{lemma}
\begin{proof}
By the equivalent formula \eqref{eqn:BDF-frac-iter-1} and the expression \eqref{eqn:solrep-frac}, we have
\begin{equation*}
\begin{aligned}
U_m^n &=- \kappa \tau^{1-\alpha} \sum_{i=1}^{n} E_\tau^{n-i}  G^{i}_m
=- \kappa \tau^{1-\alpha} \sum_{i=1}^{n} E_\tau^{n-i}
\sum_{j=0}^{N-1} \omega_{j+i}^{(\alpha)} (U_m^{N-j} -    U_{m-1}^{N-j}  ).
\end{aligned}
\end{equation*}
Now we take the $H$ norm in the above equality and apply Lemma \ref{lem:sol-op-frac} to obtain that
\begin{equation*}
\begin{aligned}
 \|U_m^n\|  &\le c\kappa \tau^{1-\alpha} \sum_{i=1}^{n} t_{n-i+1}^{\alpha-1}
 \sum_{j=0}^{N-1} |\omega_{j+i}^{(\alpha)}|~\| U_m^{N-j} -    U_{m-1}^{N-j} \|\\
  &\le c\kappa  \sum_{i=1}^{n} (n-i+1)^{\alpha-1}
 \sum_{j=0}^{N-1} |\omega_{j+i}^{(\alpha)}|~\| U_m^{N-j} -    U_{m-1}^{N-j} \|.
 \end{aligned}
\end{equation*}
Then Lemma \ref{lem:coeff} indicates that
\begin{equation}\label{eqn:err-frac-1}
\begin{aligned}
 \|U_m^n\| &\le c\kappa \sum_{i=1}^{n} (n-i+1)^{\alpha-1} \sum_{j=0}^{N-1}  (j+i)^{ -\alpha-1 }  \| U_m^{N-j} -    U_{m-1}^{N-j} \| \\
 &= c\kappa \sum_{j=0}^{N-1} \| U_m^{N-j} -    U_{m-1}^{N-j} \|  \sum_{i=1}^{n} (n-i+1)^{\alpha-1} (j+i)^{ -\alpha-1 } \\
 &\le c\kappa \tau^\alpha \sum_{j=0}^{N-1} t_{j+1}^{-\alpha} \| U_m^{N-j} -    U_{m-1}^{N-j} \| \sum_{i=1}^{n} (n-i+1)^{\alpha-1} (j+i)^{-1 }\\
 &\le c\kappa t_n^{\alpha-1} \ln(n+1) \Big(\tau\sum_{j=1}^{N} t_{j}^{-\alpha} \| U_m^{N-j+1} -    U_{m-1}^{N-j+1} \|  \Big).
\end{aligned}
\end{equation}
\ZZ{The last inequality follows from the estimate that \cite[Lemma 11]{JinZhou:iter}
$$\sum_{i=1}^{n} (n-i+1)^{\alpha-1} (j+i)^{-1 } \le c n^{\alpha-1} \ln(n+1).$$
Multiplying $\tau t_{N-n+1}^{-\alpha}$ on  \eqref{eqn:err-frac-1} and summing over $n$, we derive that for $\alpha \in (0,1)$
\begin{equation*}
\begin{aligned}
\tau \sum_{n=1}^N t_n^{-\alpha} \|U_m^{N-n+1}\| & \le   c\kappa \Big(\tau \sum_{n=1}^N t_{N-n+1} ^{-\alpha}
t_n^{\alpha-1} \ln(n+1)\Big) \Big(\tau\sum_{j=1}^{N} t_{j}^{-\alpha} \| U_m^{N-j+1} -    U_{m-1}^{N-j+1} \|  \Big)\\
&\le c\kappa \log(N)\Big( \sum_{n=1}^N (N-n+1) ^{-\alpha}
n^{\alpha-1} \Big)  \Big(\tau\sum_{n=1}^{N} t_{n}^{-\alpha} \| U_m^{N-n+1} -    U_{m-1}^{N-n+1} \|  \Big)\\
&\le c\kappa \log(N)  \Big(\tau\sum_{n=1}^{N} t_{n}^{-\alpha} \| U_m^{N-n+1} -    U_{m-1}^{N-n+1} \|  \Big)
\end{aligned}
\end{equation*}
where the constant $c$ in the second inequality depends on $\alpha$. In the last inequality, we use the fact that
\begin{equation*}
\begin{aligned}
\sum_{n=1}^N (N-n+1) ^{-\alpha} n^{\alpha-1} &\le \sum_{n=1}^N \int_{n-1}^n (N-n+1) ^{-\alpha} n^{\alpha-1} \,\d s
\le  \sum_{n=1}^N \int_{n-1}^n (N-s) ^{-\alpha} s^{\alpha-1} \,\d s \\
&= \int_{0}^N (N-s) ^{-\alpha} s^{\alpha-1} \,\d s
= \int_0^1 (1-s)^{-\alpha} s^{\alpha-1} \,\d s = B(\alpha,1-\alpha),\\
\end{aligned}
\end{equation*}
where $B(\cdot,\cdot)$ denotes the Beta function.}

Next, we apply triangle's inequality, chose $\kappa$ small enough such that $c\kappa \log(N)<1$, and hence derive that
\begin{equation}\label{eqn:err-frac-2}
\begin{aligned}
\tau \sum_{n=1}^N t_n^{-\alpha} \|U_m^{N-n+1}\|  &\le \frac{c\kappa \log(N)}{1-c\kappa \log(N)} \Big(\tau \sum_{n=1}^N t_n^{-\alpha} \|U_{m-1}^{N-n+1}\|\Big) .
\end{aligned}
\end{equation}
Finally, we define
\begin{equation}\label{kappa2}
 \gamma(\kappa) =  \frac{c\kappa \log(N)}{1-c\kappa \log(N)}.
\end{equation}
By choosing $\kappa$ such that $c\kappa \log(N) \in (0,1/2)$, we have $\gamma(\kappa)\in(0,1)$.
%
Therefore, by \eqref{eqn:err-frac-1}, we obtain that
\begin{equation*}
\begin{aligned}
 \|U_m^n\|  &\le c\kappa t_n^{\alpha-1} \log(N) \Big(\tau\sum_{j=1}^{N} t_{j}^{-\alpha} \| U_m^{N-j+1} -    U_{m-1}^{N-j+1} \|  \Big)\\
 &\le c\kappa t_n^{\alpha-1} \log(N) \Big(\tau\sum_{j=1}^{N} t_{j}^{-\alpha} \| U_m^{N-j+1}\|  +  \tau\sum_{j=1}^{N} t_{j}^{-\alpha}  \| U_{m-1}^{N-j+1} \|  \Big)\\
 &\le c\kappa t_n^{\alpha-1} \log(N) (1+\gamma(\kappa)) \tau\sum_{j=1}^{N} t_{j}^{-\alpha}  \| U_{m-1}^{N-j+1} \| \\
 &\le ct_n^{\alpha-1}   \gamma(\kappa) \Big( \tau\sum_{j=1}^{N} t_{j}^{-\alpha}  \| U_{m-1}^{N-j+1} \|\Big).
\end{aligned}
\end{equation*}
Then repeating the estimate \eqref{eqn:err-frac-2} leads to the desired result.
\end{proof}

 Then the stability result in Lemma \ref{lem:conv-frac} leads to the convergence.

 \begin{corollary}\label{cor:conv-frac-1}
 Let $U_m^n$ be the solution to the iterative scheme \eqref{eqn:BDF-frac-iter}, and $U^n$
 be the solution to the $k$-step BDF scheme  \eqref{eqn:BDF-frac}. Then we can choose a proper parameter $\kappa=O(1/\log(N))$
 in \eqref{eqn:BDF-frac-iter}, such that the following estimate holds valid:
\begin{equation*}
   \|U^n - U_m^{n}\| \le c t_n^{\alpha-1}\gamma(\kappa)^m \Big(\tau \sum_{j=1}^{N} t_j^{-\alpha}\| U^{N-j+1} - U_{0}^{N-j+1}\|\Big),\quad \forall ~~m\ge1.
\end{equation*}
 Here the convergence factor $\gamma(\kappa)$, given by \eqref{kappa2},
 might depend on $\kappa$, $\beta_0$ and $T$, but   independent of $\tau$, $n$ and $m$.
 \end{corollary}
 \begin{proof}
 To this end, we let $e_m^n = (U_m^n - U^n)$ with $1\le n\le N$, and note the fact that $\{U^n\}_{n=1}^N$ is
 the fixed point of the iteration \eqref{eqn:BDF-iter}. Therefore $e_m^n$ satisfies
\begin{equation} \label{eqn:err-heat-1}
\begin{aligned}
\bar\partial_\tau ^\alpha e_m^{n} + Ae_m^n  &= - \frac{\kappa}\tau K^n_m , \quad n = 1,2,\cdots, N,\\
e_m^{-j} &=  0 ,\quad j=0,1,\ldots,k-1.
\end{aligned}
\end{equation}
where the term $K_m^n$ is given by
\begin{equation*}
\begin{aligned}
K^n_m =
\sum_{j=0}^{N-1} \omega_{n+j} (e_m^{N-j} -    e_{m-1}^{N-j}  ).
\end{aligned}
\end{equation*}
Then the convergence estimate follows immediately from Lemma \ref{lem:conv-frac}.
\end{proof}

Combining Corollary \ref{cor:conv-frac-1} with the estimate \eqref{eqn:err-BDF-frac}, we have the following error estimate of the iterative scheme
\eqref{eqn:BDF-frac-iter}.
\begin{theorem}\label{thm:conv-frac}
Suppose that the condition \eqref{eqn:coe} and the assumption of data regularity \eqref{eqn:data-frac} hold true.
Let $U_m^n$ be the solution to the iterative algorithm \eqref{eqn:BDF-frac-iter}
with the initial guess $U_0^n = v$ for all $0 \le n\le N$, and $u$
be the exact solution to the subdiffusion equation \eqref{eqn:fde}. Then for $1\le n\le N$,  we have
\begin{equation*}
 \begin{aligned}
   \|U_m^n - u(t_n)\| &\le c(\gamma(\kappa)^m t_n^{\alpha-1} + \tau^kt_n^{-k}),\quad \text{with}~~\kappa = O(1/\log(N)),
 \end{aligned}
\end{equation*}
Here constant $c$ and the convergence factor  $\gamma(\kappa)$  given by \eqref{kappa2}
 might depend on $k$, $\kappa$, $\beta_0$, $T$, $v$ and $f$,
but they are   independent of $\tau$, $n$, $m$ and $u$.
\end{theorem}
\begin{proof}
We split the error into two parts:
$$ U_m^n - u(t_n) = (U_m^n - U^n) + (U^n - u(t_n)).  $$
The second term has the error bound by \eqref{eqn:err-BDF-frac}. Meanwhile,
via Corollary \ref{cor:conv-frac-1}, the first component converges to zero as $m\rightarrow0$ and we have the estimate
 \begin{equation*}
\begin{aligned}
  \|U^n - U_m^{n}\| &\le c t_n^{\alpha-1}\gamma(\kappa)^m \Big(\tau \sum_{j=1}^{N} t_j^{-\alpha} \| U^{N-j+1} - v\|\Big)\\
  &\le c t_n^{\alpha-1} \gamma(\kappa)^m \Big(\| v\| +  \tau \sum_{j=1}^{N} t_j^{-\alpha}\| U^{N-j+1}\|\Big)
\end{aligned}
\end{equation*}
Note that the estimate \eqref{eqn:err-BDF-frac} and the assumption of data regularity \eqref{eqn:data-frac} implies the uniform bound of $U^n$ for all $n=1,2,\ldots,N$, we obtain that
 \begin{equation*}
\begin{aligned}
   \|U^N - U_m^{N}\|  \le c_T  t_n^{\alpha-1} \gamma(\kappa)^m .
\end{aligned}
\end{equation*}
Then the desired result follows immediately.
\end{proof}

\begin{remark}
By the expression of the convergence factor $\gamma(\kappa)$ in \eqref{kappa2},
we expect that the iteration converges linearly when $c\kappa \log(N) \in (0,1/2)$, i.e., $\kappa < 1/(2c\log N)$.
Besides, it implies that the convergence rate might deteriorate slightly for a large $N$ and a fixed $\kappa$.
Surprisingly, our numerical results indicate that the iteration converges robustly even for relatively large $\kappa$ (cf. Figure \ref{fig:eg2-kappa}),
and the step number $N$ seems not affect the convergence rate  (cf. Figure \ref{fig:eg2-tau}).
\end{remark}

\section{Numerical Tests}\label{sec:numerics}
In this section, we present some numerical results to illustrate and
complement our theoretical findings.  The computational domain is the
unit interval $\Omega = (0,1)$ for the Example 1 and Example 2, and
the unit square $\Omega = (0,1)^2$ for the Example 3.  In space, it is
discretized with piecewise linear Galerkin finite element method on a
uniform mesh with mesh size $h$ for one-dimensional problems.
\ZZ{For two-dimensional problems, we compute numerical solutions on a uniform triangulation with mesh size $h$.}
We focus on the convergence behavior
of the iterative solver to the BDF$k$ solution, since the temporal
convergence of BDF$k$ scheme has been theoretically studied and numerically examined in
\cite{JinLiZhou:SISC2017}. That is, with the fixed time step size
$\tau = T/N$, we measure the error in the $m$-th iteration $$e_m^N :=
\|U_m^N - U^N\|_{L^2(\Omega)},$$ where we take the BDF$k$ solution
$U^N$ as the reference solution.

\subsection{Numerical results for normal diffusion} ~

\begin{example}[1D diffusion equation]\label{Example1}
{\upshape We begin with the following one-dimensional normal diffusion
equation:
\begin{equation}\label{eqn:linear-parabolic}
    \left\{\begin{aligned} \partial_t  u  -   \partial_{xx} u  &=
      f(x,t),&&
    \mbox{in}\,\,\,\Omega\times(0,T],\\
 u &=0,&&
\mbox{on}\,\,\,\partial\Omega\times(0,T] , \\
    u(0) &= v , && \mbox{in}\,\,\,\Omega,
   \end{aligned} \right.
\end{equation}
where $\Omega=(0,1)$ and $T = 0.5$. We consider the following problem data
\begin{equation*}
  v(x) = \chi_{(0,\frac{1}{2})}(x)\quad \text{and} \quad f(x,t) = e^t \cos(x),
\end{equation*}
where $\chi$ denotes the characteristic function.

First, we check the performance of the algorithm for different
PinT BDF$k$ schemes. Taking $\kappa = 0.5$ and $\tau =
T/100$, the numerical results using the Algorithm \ref{alg:heat} with
different orders of BDF schemes are presented in Table
\ref{tab:BDFk-heat}. It can be seen that all the PinT
BDF$k$ schemes ($k \leq 6$) converge fast  in a similar manner.  In
what follows, we take the PinT BDF$3$ scheme to check
the influence of different $N$ (or $\tau$) and $\kappa$.

\begin{table}[htb!]
\caption{Example \ref{Example1}:  $e^N_m$ for $T = 0.5$, $\tau = T/100$,
  $h = 1/1000$ and $\kappa = 0.5$.}
\label{tab:BDFk-heat}
\centering
\begin{tabular}{|c|cccccc|}
\hline
 $m\backslash$ BDF$k$ & $k=1$ & $k=2$ & $k=3$ & $k=4$ & $k=5$ & $k=6$\\
\hline
0 & 1.20e-01 & 1.20e-01 & 1.20e-01 & 1.20e-01 & 1.20e-01 & 1.20e-01\\

1 &4.88e-04 & 4.43e-04 & 4.44e-04 & 4.44e-04 & 4.44e-04 & 4.44e-04\\

2 &1.98e-06 & 1.59e-06 & 1.60e-06 & 1.60e-06 & 1.60e-06 & 1.60e-06\\

3 & 8.05e-09 & 5.72e-09 & 5.79e-09 & 5.79e-09 & 5.79e-09 & 5.78e-09\\

4 & 2.81e-11 & 2.37e-11 & 1.98e-11 & 1.74e-11 & 1.92e-11 & 1.99e-11\\

5 & 4.76e-12 & 3.09e-12 & 1.11e-12 & 3.54e-12 & 1.81e-12 & 1.10e-12\\
\hline
\end{tabular}
\end{table}

Taking $\kappa = 0.5$, we report the convergence histories with
different time step sizes in Figure \ref{fig:eg1-tau}. It is seen that
the converence rate is independent of $\tau$, which agrees well with
the Corollary \ref{cor:conv-heat-1}. In Figure \ref{fig:eg1-kappa} we
plot the convergence histories with different $\kappa$. It can be seen
that, with the decrease of $\kappa$, the convergence becomes faster,
which is in agreement with the convergence rate in theory
\eqref{eq:gamma-heat}. On the other hand, the smaller  $\kappa$ will
lead to larger roundoff error, as we proved in Section \ref{ssec:dev-heat}.  Hence, one needs to choose $\kappa$
properly to balance the convergence rate and roundoff error.

\begin{figure}[!htbp]
\centering
\captionsetup{justification=centering}
\subfloat[$\kappa = 0.5$, influence of $\tau = T/N$]{
  \includegraphics[width=0.4\textwidth]{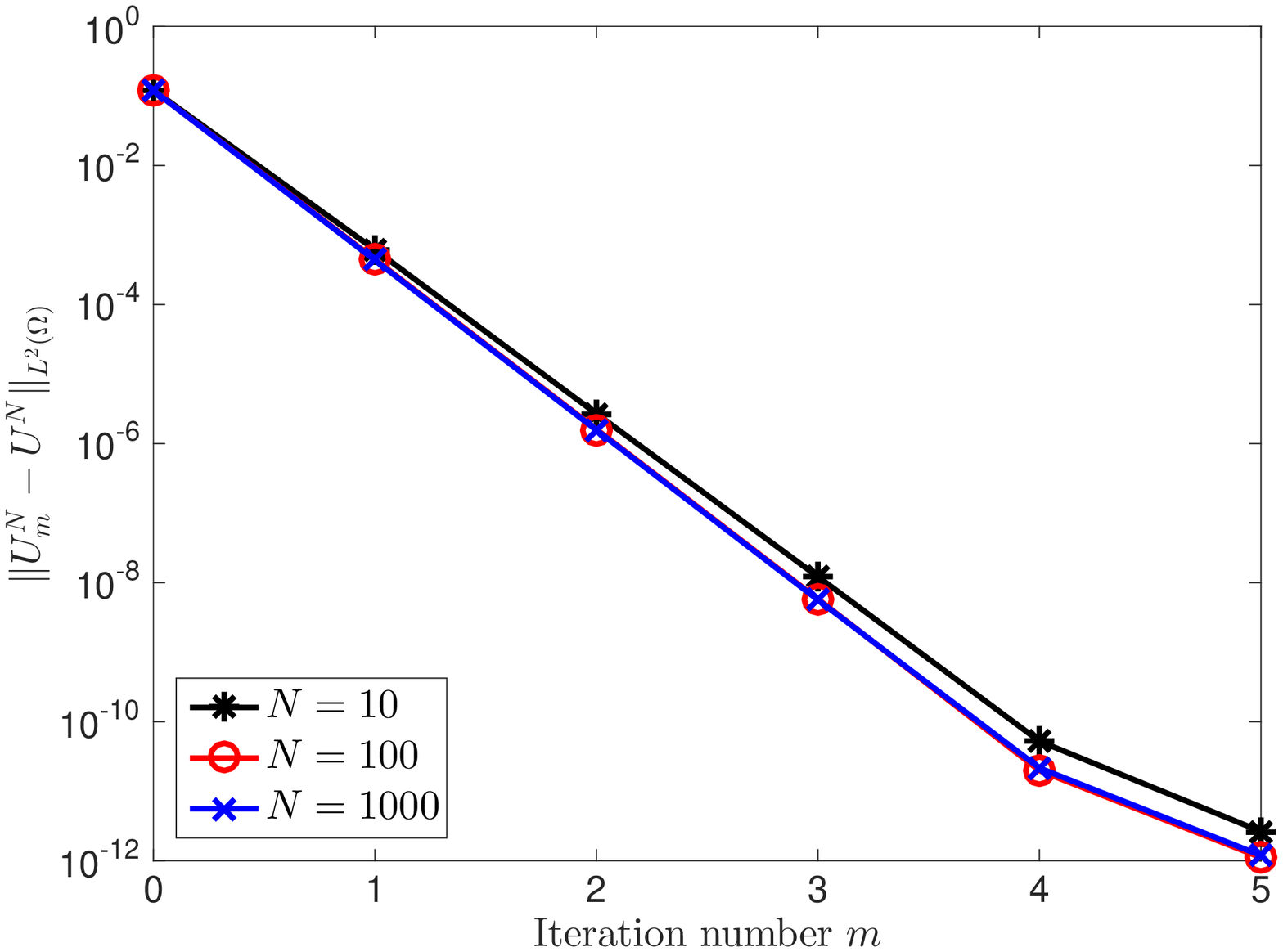}
  \label{fig:eg1-tau}
} %
\subfloat[$\tau=T/200$, influence of $\kappa$]{
  \includegraphics[width=0.4\textwidth]{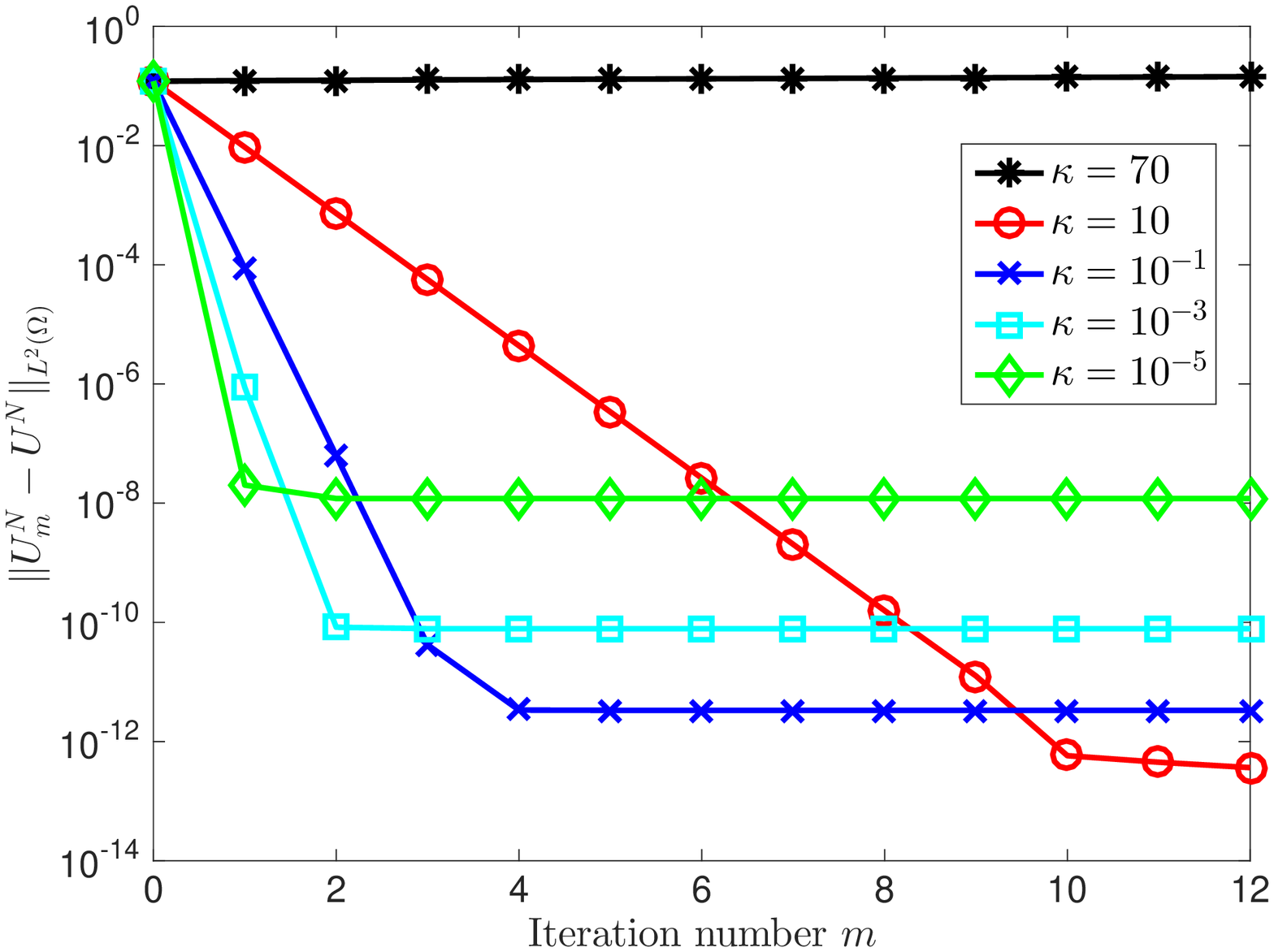}
  \label{fig:eg1-kappa}
}
\caption{PinT BDF3 for Example \ref{Example1}:  $T  = 0.5$ and $h = 1/1000$.}
\label{fig:eg1}
\end{figure}
}
\end{example}

%

\subsection{Numerical results for subdiffusion}
In this subsection, we test the performance of the algorithm for the
subdiffusion problem in both 1D and 2D:
\begin{equation}\label{eqn:linear-frac}
    \left\{\begin{aligned} \Dal  u  -  \Delta u  &= f(x,t),&&
    \mbox{in}\,\,\,\Omega\times(0,T],\\
 u &=0,&&
\mbox{on}\,\,\,\partial\Omega\times(0,T] , \\
    u(0) &= v , && \mbox{in}\,\,\,\Omega.
   \end{aligned} \right.
\end{equation}

\vskip5pt
\begin{example}[1D subdiffusion equation]\label{Example2}
{\upshape
In the one-dimensional problem, the computational domain is $\Omega =
(0,1)$ with equally spaced mesh. The mesh size is set to be $h = 1/M$
with $M = 1000$. We consider the following problem data
$$
v = \delta_{1/2}(x) \quad \text{and} \quad f(x,t) = 0.
$$
Here the initial data is the Dirac-delta measure concentrated at $x=\frac12$, which only belongs to
$H^{-\frac12-\epsilon}(\Omega)$ for any $\epsilon>0$. In the computation, the initial value is set to
be the $L^2-$projection of delta function; see some details in  \cite{Jin:weak}.

Even though the initial condition is very weak,
the inverse inequality and the analysis in \cite{Jin:weak, JinLiZhou:SISC2017}
implies the error estimate
$$  \| U_n - u(t_n) \| \le c (\tau^k h^{-\frac12-\epsilon} t_n^{-k}  +  h^{\frac32-\epsilon} t_n^{-\alpha}),$$
and also the following convergence result
$$  \| U_n - U_n^m \| \le c t_n^{\alpha-1} h^{-\frac12-\epsilon} \gamma(\kappa)^m$$
with the same $\gamma(\kappa)$ defined in \eqref{kappa2}.


\ZZ{
Similar to the normal diffusion, the performance of all the
PinT BDF$k$ ($k \leq 6$) have the same convergence
profile, see Table \ref{tab:BDFk-frac}. Moreover, we check the
influence of different $N$ (or $\tau$) and $\kappa$ by using the
PinT BDF3 scheme.
Our theoretical result \eqref{kappa2} indicates that the convergence factor is
\begin{equation*}
 \gamma(\kappa) =  \frac{c\kappa \log(N)}{1-c\kappa \log(N)},
\end{equation*}
with some generic constant $c>1$. So we expect that the iteration converges linearly when $c\kappa \log(N) \in (0,1/2)$, i.e., $\kappa < 1/(2c\log N)$.
Besides, it implies that the convergence rate might deteriorate slightly for a large $N$  and a fixed $\kappa$.
Surprisingly, our numerical results indicate that the iteration converges robustly even for relatively large $\kappa$ (cf. Figure \ref{fig:eg2-kappa}),
and the step number $N$ seems not affect the convergence rate  (cf. Figure \ref{fig:eg2-tau}).
From Figure
\ref{fig:eg2-kappa}, we observe that the influence of $\kappa$ is
similar to the normal diffusion case: the smaller $\kappa$ will lead
to faster convergence rate but worse roundoff error. In practice,
 the choice $\kappa\approx 10^{-1}$  leads to an acceptable roundoff error ($\approx 10^{-11}$),
and meanwhile the convergence is fast. }

\begin{table}[htb!]
\caption{Example \ref{Example2}:   $e^N_m$ for $T = 0.1$, $\alpha = 0.5$, $\tau =
  T/100$,
  $h = 1/1000$ and $\kappa = 0.1$.}
\label{tab:BDFk-frac}
\centering
\begin{tabular}{|c|cccccc|}
\hline
 $m\backslash$ BDF$k$ & $k=1$ & $k=2$ & $k=3$ & $k=4$ & $k=5$ & $k=6$\\
\hline
0 & 2.46e-01 & 2.46e-01 & 2.46e-01 & 2.46e-01 & 2.46e-01 & 2.46e-01 \\

1 &6.31e-04 & 6.28e-04 & 6.28e-04 & 6.28e-04 & 6.28e-04 & 6.28e-04\\

2 &2.88e-06 & 2.85e-06 & 2.85e-06 & 2.84e-06 & 2.85e-06 & 2.85e-06\\

3 & 1.34e-08 & 1.32e-08 & 1.32e-08 & 1.32e-08 & 1.33e-08 & 1.31e-08\\

4 & 8.12e-11 & 4.79e-11 & 4.98e-11 & 8.31e-11 & 4.27e-11 & 1.37e-10\\

5 & 1.88e-11 & 8.47e-12 & 1.86e-11 & 1.61e-11 & 3.44e-11 & 1.50e-10\\
\hline
\end{tabular}
\end{table}

\begin{figure}[!htbp]
\centering
\captionsetup{justification=centering}
\subfloat[$\kappa = 0.1$, influence of $\tau = T/N$]{
  \includegraphics[width=0.35\textwidth]{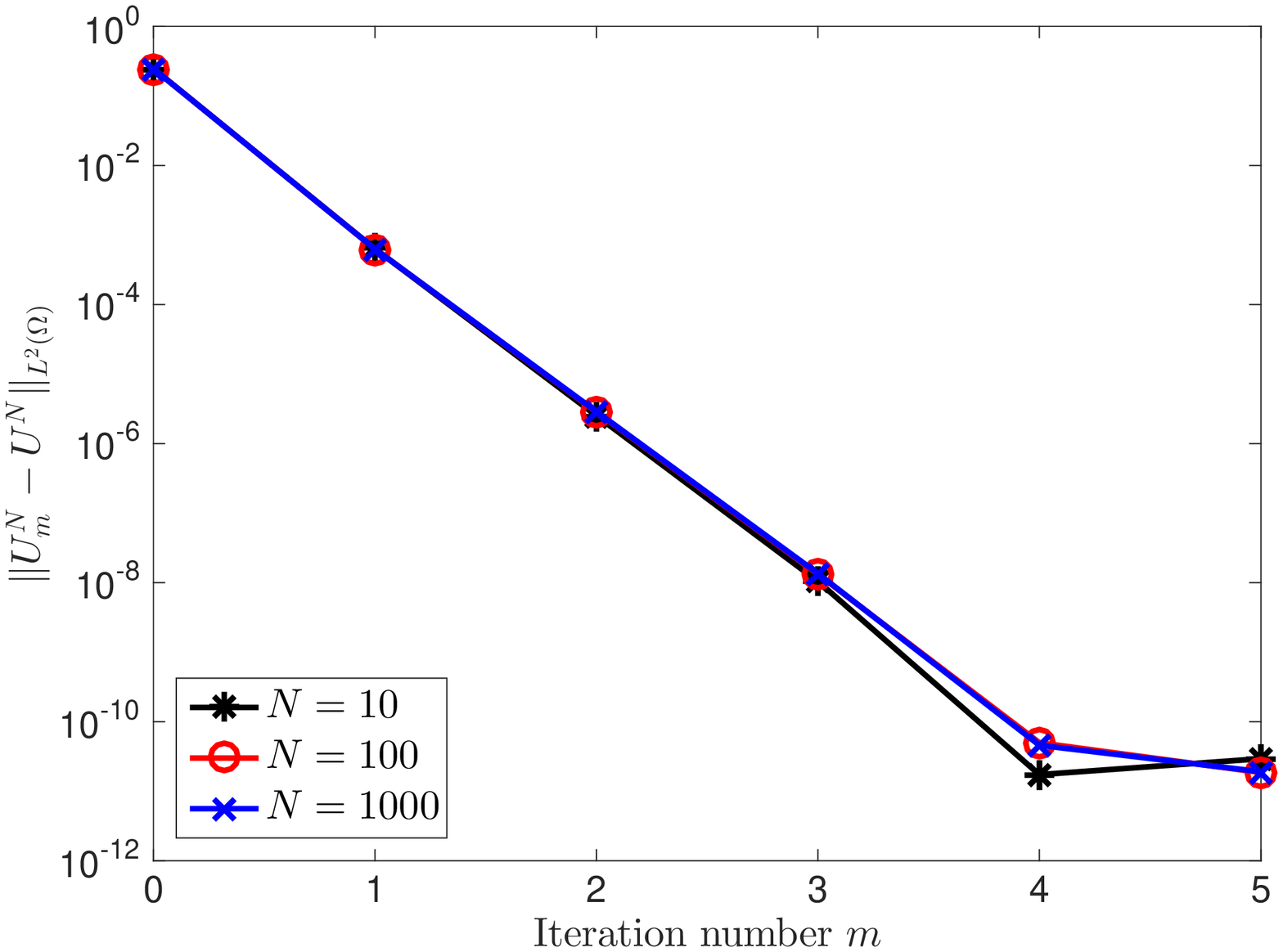}
  \label{fig:eg2-tau}
} %
\subfloat[$\tau=T/100$, influence of $\kappa$]{
  \includegraphics[width=0.35\textwidth]{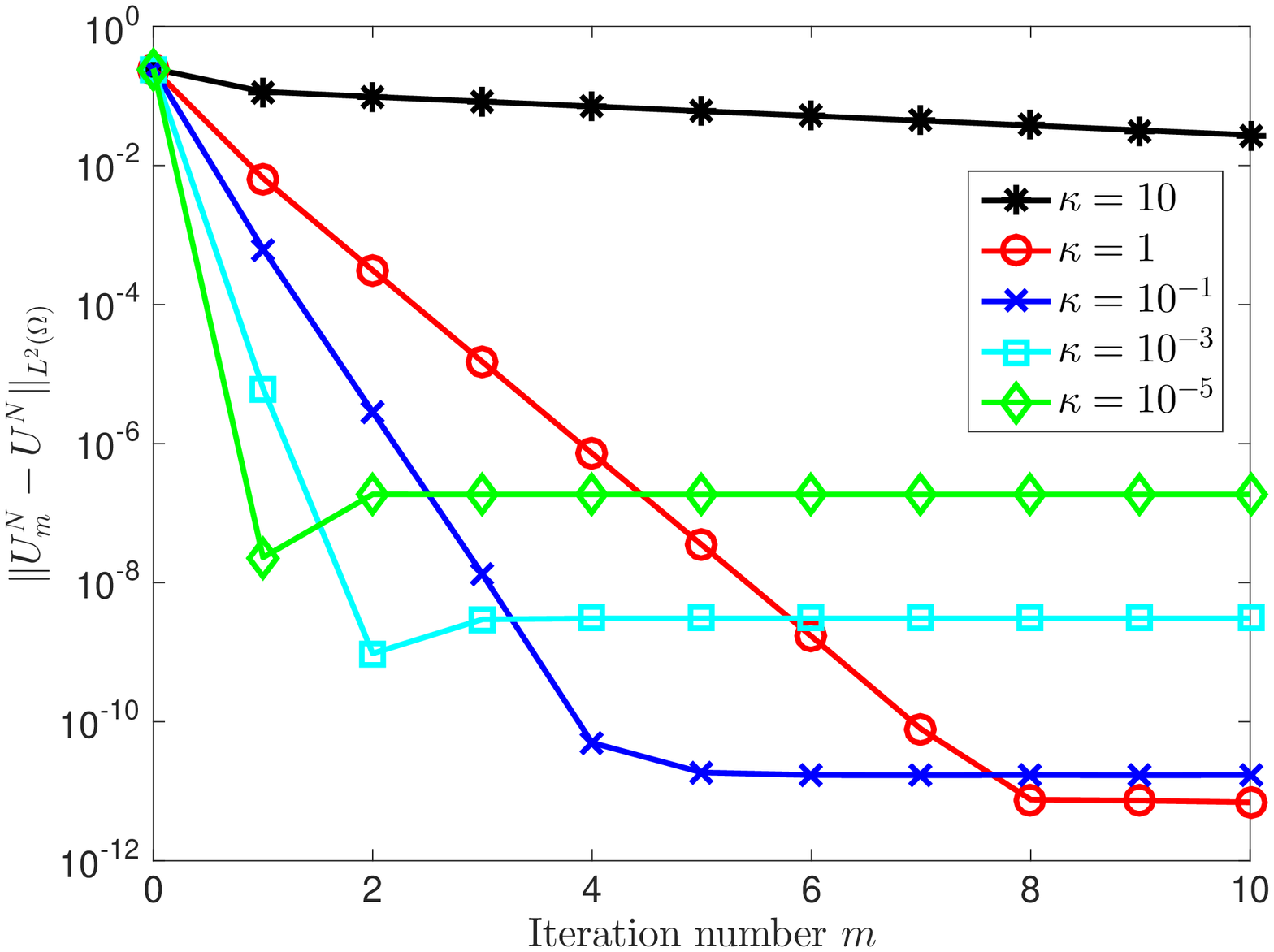}
  \label{fig:eg2-kappa}
}
\caption{PinT BDF3 for Example \ref{Example2}: $\alpha = 0.5$, $T  = 0.1$ and $h = 1/1000$.}
\label{fig:eg2}
\end{figure}
}
\end{example}

\vskip5pt
\begin{example}[2D subdiffusion equation]\label{Example3}
{\upshape In this example, the spatial discretization is taken on the
uniform triangulation of $\Omega=(0,1)^2$. We consider the following
problem data
\begin{equation*}
  v(x) = \chi_{(0,\frac12)\times(0,\frac12)}(x)\quad \text{and} \quad
  f(x,t) = \cos(t)\chi_{(\frac12,1)\times(\frac12,1)}(x).
\end{equation*}

The numerical solutions are computed on a uniform triangular mesh with $h = 10^{-2}$.
We also observe that $\kappa$ needs to be properly chosen to
balance the convergence rate and roundoff error, see Figure
\ref{fig:eg3-sol}.

\begin{figure}[!htbp]
\centering
\captionsetup{justification=centering}
  \includegraphics[width=0.4\textwidth]{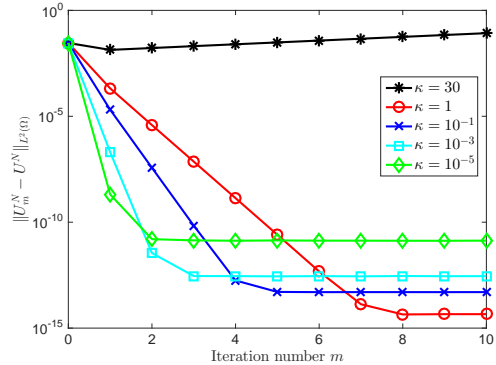}
\caption{PinT BDF3 for Example \ref{Example3}:
  convergence histories for $\alpha = 0.5$, $\tau = 10^{-3}$ and $h =
  10^{-2}$.}
\label{fig:eg3-sol}
\end{figure}

\begin{figure}[!htbp]
\centering
\captionsetup{justification=centering}
  \subfloat[$T = 0.01$, influence of $\alpha$]{
  \includegraphics[width=0.32\textwidth]{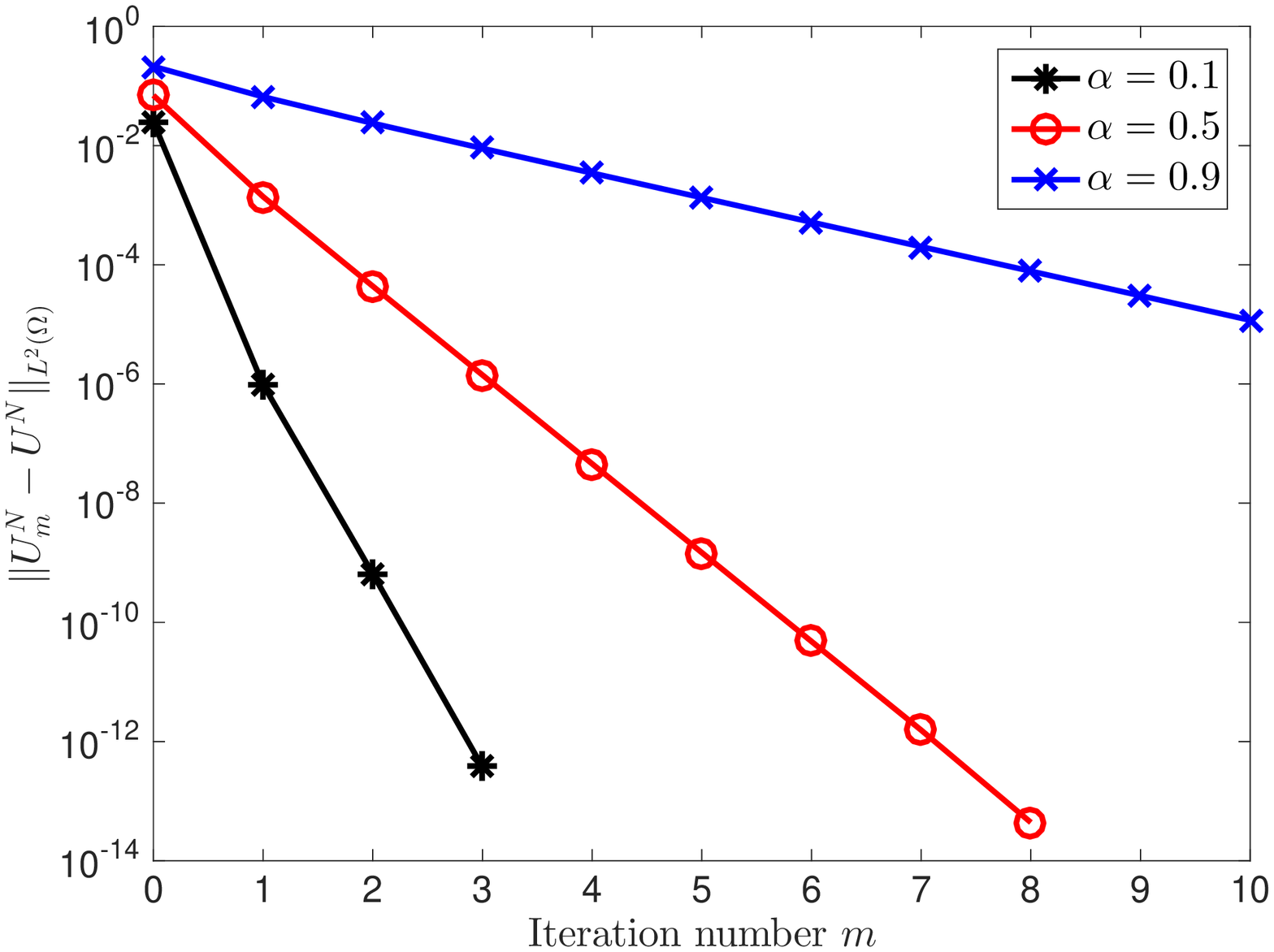}
  \label{fig:eg3-alpha1}
} %
  \subfloat[$T = 0.1$, influence of $\alpha$]{
  \includegraphics[width=0.32\textwidth]{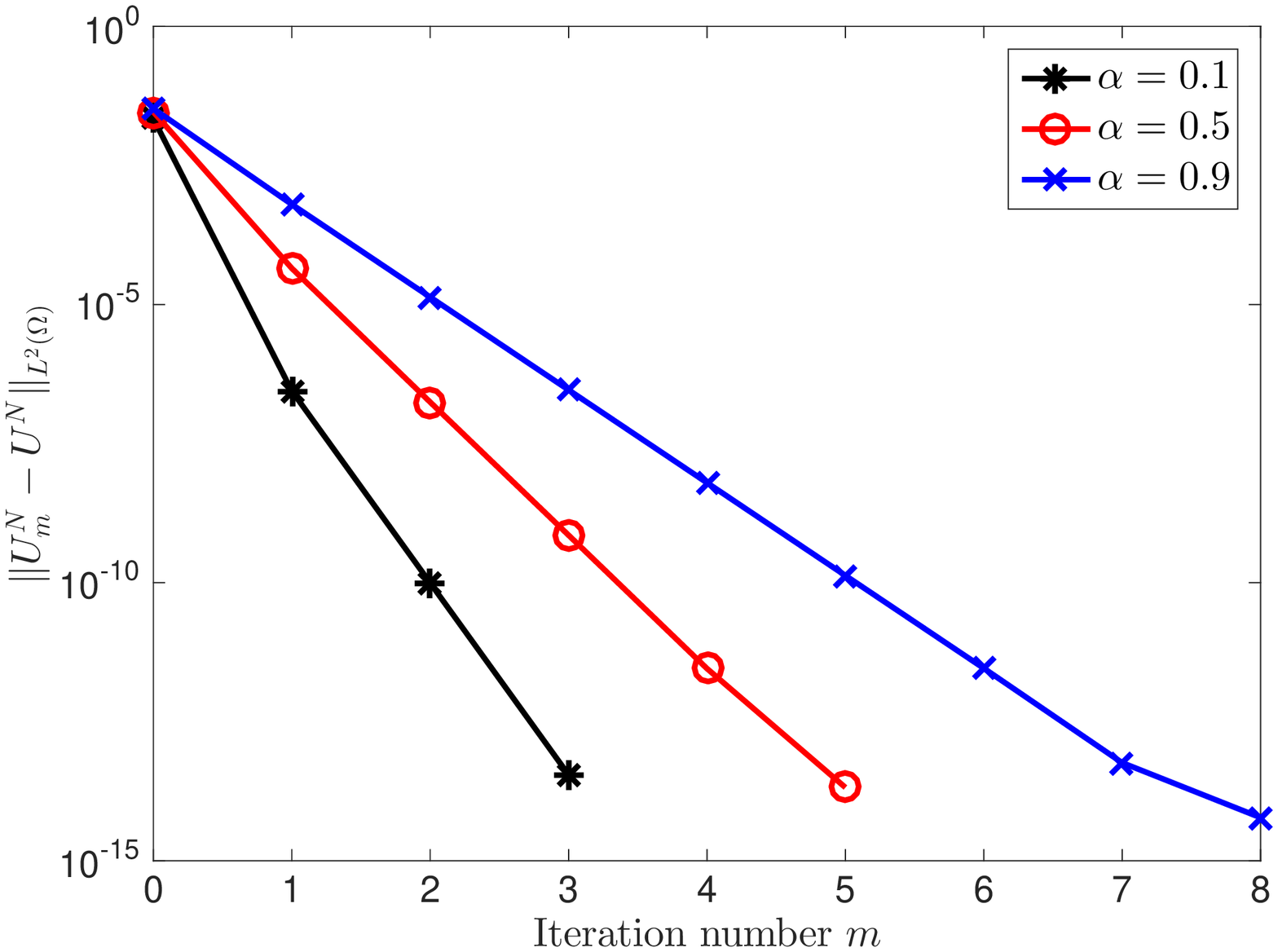}
  \label{fig:eg3-alpha2}
} %
  \subfloat[$T = 1$, influence of $\alpha$]{
  \includegraphics[width=0.32\textwidth]{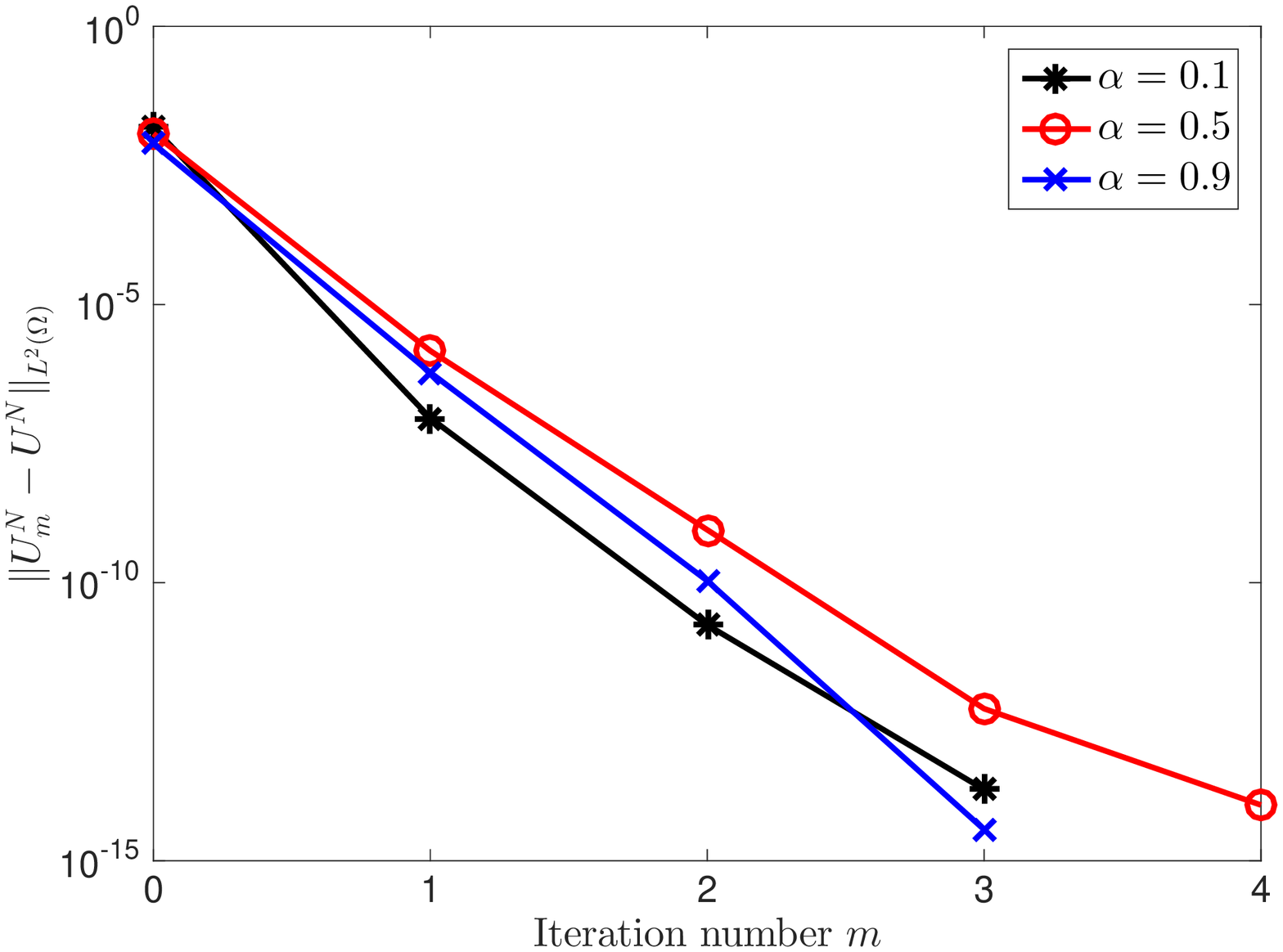}
  \label{fig:eg3-alpha3}
}
\caption{PinT BDF3 for Example \ref{Example3}: $\kappa = 1/\log(N)$, $\tau =
  10^{-3}$ and $h = 10^{-2}$.}
\label{fig:eg3-alpha}
\end{figure}

\begin{figure}[!htbp]
\centering
\captionsetup{justification=centering}
  \subfloat[$\alpha = 0.1$, influence of $T$]{
  \includegraphics[width=0.32\textwidth]{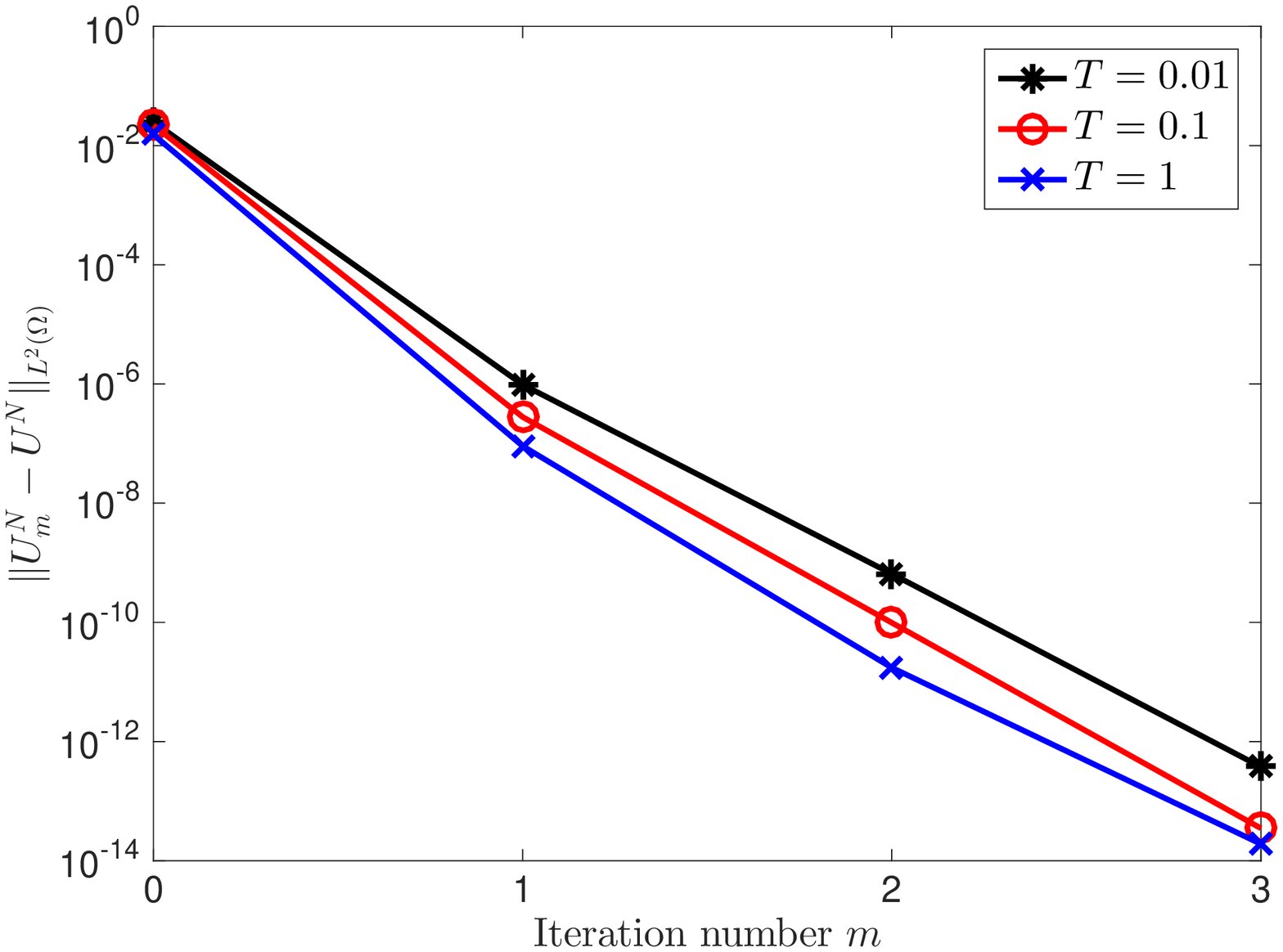}
  \label{fig:eg3-T1}
} %
  \subfloat[$\alpha = 0.5$, influence of $T$]{
  \includegraphics[width=0.32\textwidth]{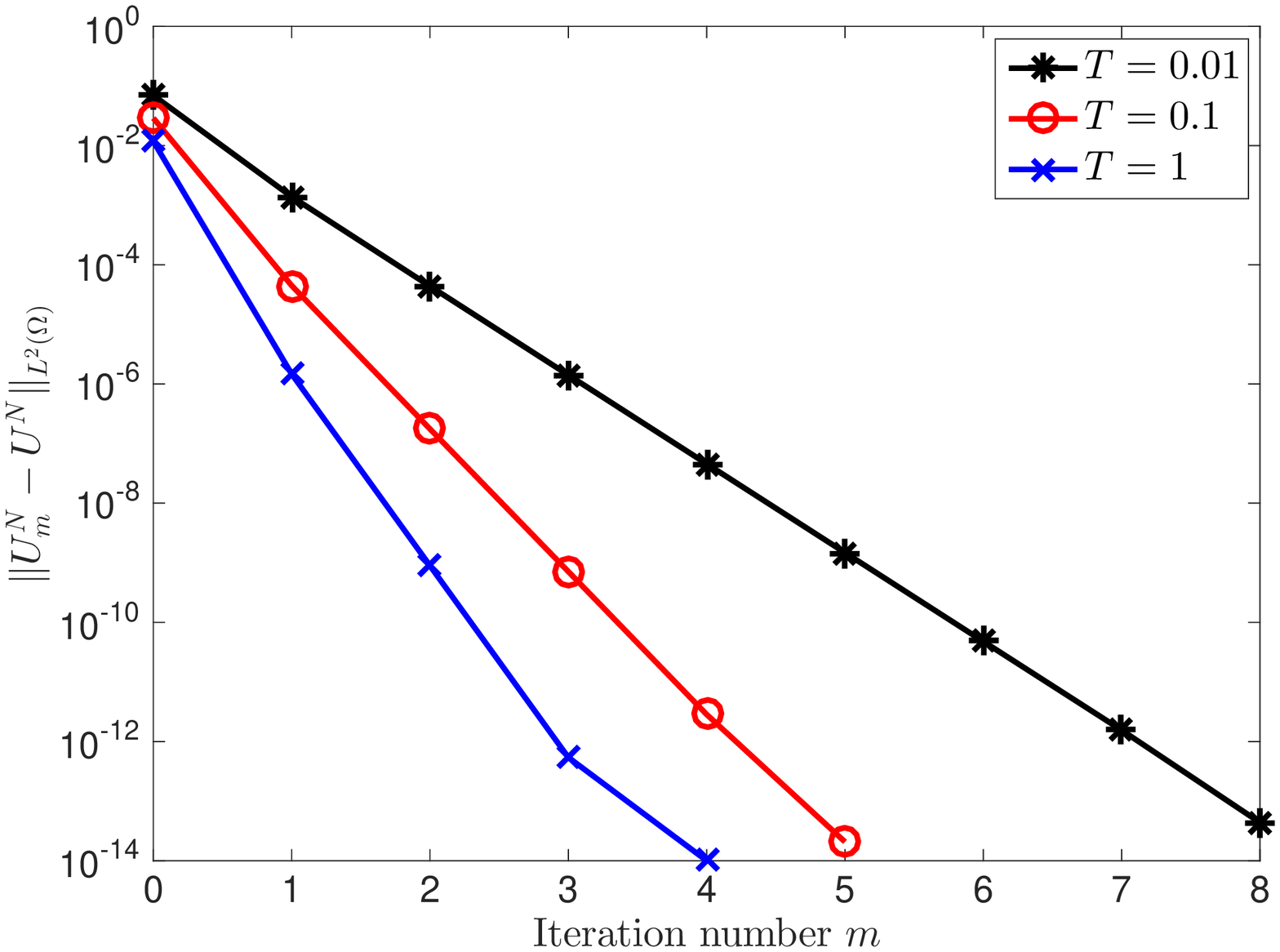}
  \label{fig:eg3-T2}
} %
  \subfloat[$\alpha = 0.9$, influence of $T$]{
  \includegraphics[width=0.32\textwidth]{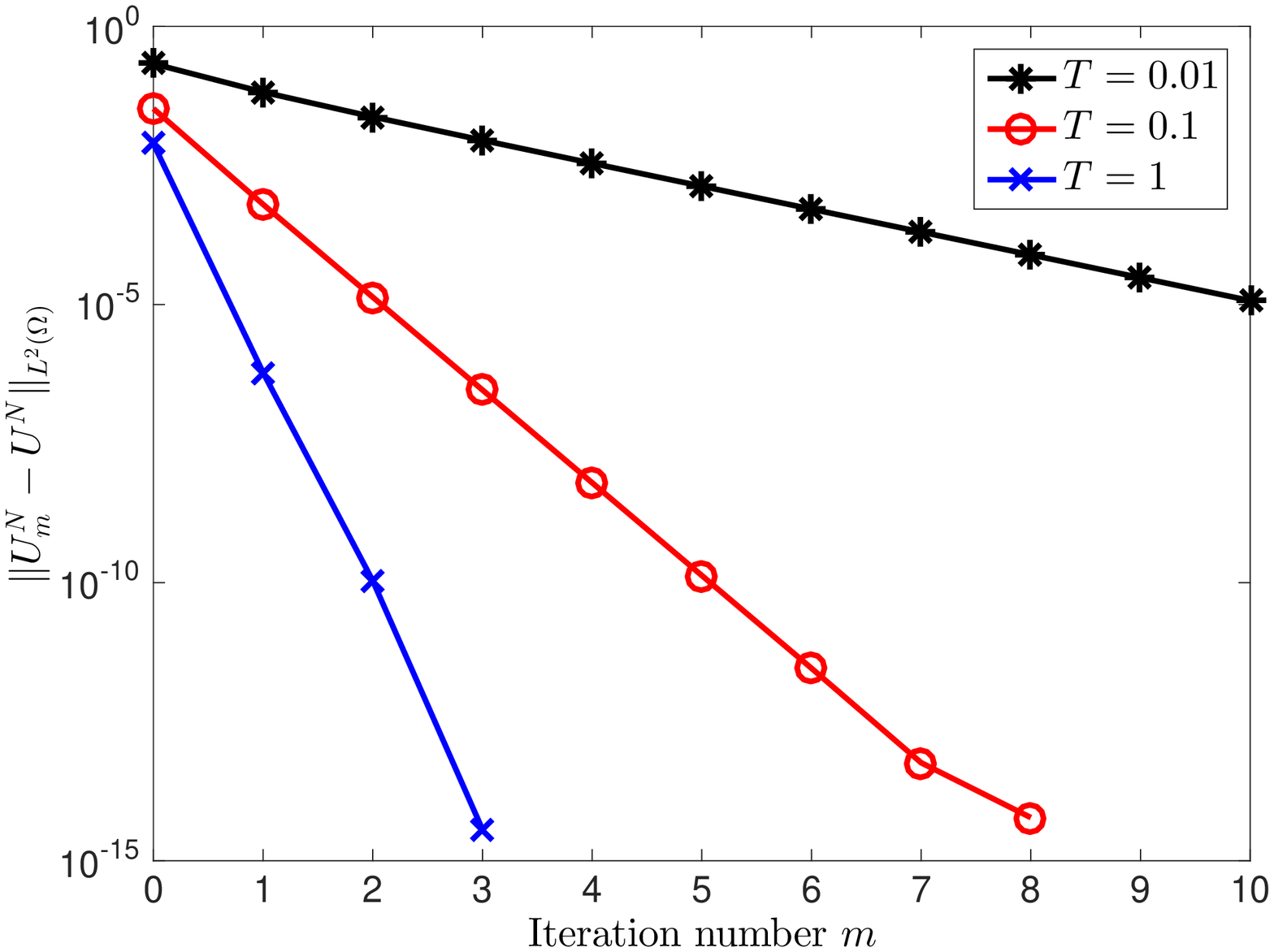}
  \label{fig:eg3-T3}
} %
\caption{PinT BDF3 for Example \ref{Example3}: Influence
  of $T$ for the convergence rate, $\kappa = 1/\log(N)$, $\tau =
  10^{-3}$ and $h = 10^{-2}$.}
\label{fig:eg3-T}
\end{figure}

Recall that, in the analysis, there is a generic constant $c$ in the
convergence rate \eqref{kappa2}, which depends on the fractional order
$\alpha$ and $T$. We numerically check these dependences and
present the results in Figure \ref{fig:eg3-alpha} and
\ref{fig:eg3-T}, respectively.  Taking $\kappa = 1/ \log(N)$, we observe the faster convergence rate with smaller $\alpha$ when $T$ is small, see Figure \ref{fig:eg3-alpha1}. With the increase of $T$, these difference is getting smaller. Further, we
observe the faster convergence rate with greater $T$ for various $\alpha$ in Figure \ref{fig:eg3-T}, which shows the
significant advantage of the proposed method for long-time simulation.
}
\end{example}

\subsection{Extension to nonlinear problems}
In this part, we shall briefly discuss  a possible application of the time-parallel
algorithm to the semilinear (sub)diffusion problem (with
$\alpha\in(0,1]$):
\begin{equation}\label{eqn:fde-nonlinear}
\left\{
\begin{aligned}
    \Dal u(t)  +  Au(t) + g(u(t))&= f(t), \quad \text{for all}~~ t\in(0,T],\\
    u(0)&=v.
\end{aligned}
\right.
\end{equation}
To numerically solve \eqref{eqn:fde-nonlinear}, we follow the similar idea introduced in Section \ref{sec:frac}
and consider the modified CQ-BDF$k$ scheme:
\begin{equation}\label{eqn:BDF-CQ-m-r}
\begin{aligned}
\bar \partial_\tau^\alpha U^n-\Delta U^n + g(U^n) &= \bar{f}_n, \quad &&1\leq n\leq N,\\
U^n&=v,\quad &&n\le 0,
\end{aligned}
\end{equation}
where
$
\bar{f}_n := f(x,t_n) + a_n^{(k)}(f(x,0) + \Delta v - g(v) ) + \sum_{\ell=1}^{k-2} b_{\ell,n}^{(k)}\tau^{\ell} \partial_t^{\ell}f(x,0)
$.
If $\alpha=1$, it reduces to a modified BDF$k$ scheme for the classical semilinear parabolic equations.
It was proved in \cite[Theorem 3.4]{WangZhou:2020} that
\begin{equation}\label{eqn:err-nonlinear}
\| U^n - u(t_n) \|_{H} \le c_T  t_n^{\alpha - \min(k,1+2\alpha-\epsilon)} \tau^{\min(k,1+2\alpha-\epsilon)}.
\end{equation}
for arbitrarily small $\epsilon$.

In order to solve the numerical solution in a time-parallel manner,
we consider a modified Newton's iteration to linearize the problem: for integer $\ell \ge1$, we compute $U_{\ell}^n=U_{\ell-1}^n+W_\ell^n$ where $W_\ell^n$ satisfies homogeneous Dirichlet boundary condition and
\begin{equation}\label{eqn:BDF-CQ-m-r-N}
\begin{aligned}
(\bar \partial_\tau^\alpha  -\Delta) W_\ell^n- g'(\overline U_{\ell-1}) W_\ell^n&= \bar{f}_n - (\bar \partial_\tau^\alpha-\Delta) U_{\ell-1}^n - g(U_{\ell-1}^n),
\quad &&1\leq n\leq N,\\
W_\ell^n&=0,\quad &&n\le 0.
\end{aligned}
\end{equation}
Here  $\overline  U_{\ell-1}$ denotes an average of $U_{\ell-1}^n$ in all levels, defined as
$$ \overline  U_{\ell-1} : = \frac1N\sum_{n=1}^N U_{\ell-1}^n.  $$
Then, for each iteration, we shall solve the linear system \eqref{eqn:BDF-CQ-m-r-N}
with a time-independent coefficient. 
Therefore, we can apply the strategy in Sections \ref{sec:heat} and \ref{sec:frac}, i.e., applying waveform relaxation to derive
an iterative solver: with $W_{\ell,0}^n=0$ and $m=1,2,\ldots$, for given $W_{\ell,m-1}^n$, we compute $W_{\ell,m}^n$ such that
\begin{equation}\label{eqn:BDF-CQ-m-r-N-w}
\begin{aligned}
(\bar \partial_\tau^\alpha  -\Delta) W_{\ell,m}^n- g'(\overline U_{\ell-1}) W_{\ell,m}^n&=\bar{f}_n - (\bar \partial_\tau^\alpha-\Delta) U_{\ell-1}^n - g(U_{\ell-1}^n),
\quad && n = 1,2,\ldots, N,\\
W_{\ell,m}^{-n} &= \kappa(W_{\ell,m}^{ {N-n}} - W_{\ell,m-1}^{ {N-n}}),\quad &&n=0, 1, \ldots, N+1,\\
W_{\ell,m}^n&=0,\quad &&n\le -N.
\end{aligned}
\end{equation}
This is a periodic-like system and hence it could be solved in parallel by diagonalization technique. We describe the complete iterative algorithm in
Algorithm \ref{alg:nonlinear}.

\begin{algorithm}[hbt!]
  \caption{~~PinT BDF$k$ scheme for nonlinear models.\label{alg:nonlinear}}
  \begin{algorithmic}[1]
    \STATE Initialize $U_0^n = v$ for all $n\in \mathbb{Z}$, and set $\ell=0$.
    \FOR{$\ell =1,\ldots,L$}
     \STATE   Initialize $W_{\ell,0}^n = 0$ for all $n\in \mathbb{Z}$, and set $m=0$.
    \FOR{$m=1,\ldots,m_\ell$}
       \STATE Solve $W_{\ell,m}^n$ satisfying \eqref{eqn:BDF-CQ-m-r-N-w} in a time parallel manner using Algorithm \ref{alg:heat} or \ref{alg:frac}.
       \STATE Check the stopping criterion of waveform relaxation.
      \ENDFOR
      \STATE Update $U_{\ell}^n = U_{\ell-1}^n + W_{\ell,m}^n.$
    \STATE Check the stopping criterion of Newton's iteration.
    \ENDFOR
  \end{algorithmic}
\end{algorithm}

\begin{example}[Allen--Cahn equations] \label{Example4}
{\upshape
Taking $A = -\Delta$, $H=L^2(\Omega)$, $V = H_0^1(\Omega)$ and
$g(u) = \frac{1}{\varepsilon^2}(u^3 - u)$ in \eqref{eqn:fde-nonlinear}, we obtain the nonlinear problem with $\alpha\in(0,1]$:
\begin{equation}\label{AC}
\left\{
 \begin{aligned}
\partial_t^\alpha u-\Delta u + \frac{1}{\varepsilon^2}(u^3 - u) &= f(x,t)&& \text{ in } \Omega\times (0,T),\\
u &=0 &&\text{ on  }  \partial\Omega\times (0,T),\\
u(0)&=v &&\text{ in }\Omega.
\end{aligned}
\right.
\end{equation}

In case that $\alpha=1$ and $f(x,t) = 0$, the model is called Allen--Cahn equation, which is a popular phase-field model,
introduced in \cite{AC} to describe the motion of anti-phase boundaries in crystalline solids.
In the context, $u$ represents the concentration of one of the two metallic components of the alloy and
the parameter $\varepsilon$ involved in the nonlinear term  represents the interfacial width, which is small compared to the
characteristic length of the laboratory scale; see also
\cite{anderson1998diffuse,chen2002phase,yue2004diffuse}
for some applications and \cite{DuYangZhou:2020} for some discussion for fractional models with $\alpha\in(0,1)$.

We shall investigate the numerical performance of Algorithm \ref{alg:nonlinear} on the domain
$\Omega = (0,1)$ with equally spaced mesh. The exterior force $f(x,t)$ is chosen such that
the exact solution yields $u = \frac{t^2}{\Gamma(3)} \sin(2\pi x)$.

First of all, we test the nonlinear problem \eqref{AC} with $\varepsilon = 1$ (mild nonlinearity),
and  report the error of Newton's iteration, i.e. $e^N_\ell = \|U_\ell^N - U^N \|_{L^2(\Omega)}.$
In the computation, we 
 and chose the stopping criteria of inner iteration (waveform relaxation) as
 $$\|W_{\ell,m}^n - W_{\ell,m-1}^n\|_\infty < 1\times 10^{-12}\quad\text{for all}\quad \ell\in \mathbb{N}^+.$$
The numbers of inner iteration  are listed in the bracket.
Invoking the error estimate \eqref{eqn:err-nonlinear}, we report the numerical results for
$\alpha = 0.25$ (with BDF$1$ and BDF$2$) and $\alpha = 0.75$ (with BDF$1$, BDF$2$ and BDF$3$) in Table \ref{tab:fde-nonlinear}.
Numerical results in Table \ref{tab:fde-nonlinear} indicate that the
inner iteration (waveform relaxation) converges robustly and quickly for the linearized system \eqref{eqn:BDF-CQ-m-r-N}, and
the modified Newton's iteration
converges fast for both cases ($\alpha = 0.25$ and $\alpha=0.75$), so does the Algorithm \ref{alg:nonlinear}.


\begin{table}[htb!]
\caption{Example \ref{Example4}:   $e^N_\ell$ with $T = 0.4$, $\tau = T/100$, $h = 1/1000$, $\kappa = 0.1$, and $\varepsilon = 1$.
 }
\label{tab:fde-nonlinear}
\centering
\subfloat[$\alpha = 0.25$]{
\begin{tabular}{|c|cc|}
\hline
 $\ell \backslash$ BDF$k$ & $k=1$ & $k=2$ \\
\hline
0 & 6.30e-02 & 6.30e-02 \\

1 &9.27e-06$(5)$ & 9.27e-06$(5)$ \\

2 &3.64e-09$(4)$ & 3.63e-09$(4)$ \\

3 & 1.42e-12$(3)$ & 1.42e-12$(3)$ \\
\hline
\end{tabular}
} \quad %
\subfloat[$\alpha = 0.75$]{
\begin{tabular}{|c|ccc|}
\hline
 $\ell \backslash$ BDF$k$ & $k=1$ & $k=2$ & $k=3$ \\
\hline
0 & 5.94e-02 & 5.94e-02 & 5.94e-02 \\

1 &6.83e-06$(5)$ & 6.80e-06$(5)$ & 6.80e-06$(5)$ \\

2 &1.95e-09$(4)$ & 1.92e-09$(4)$ & 1.92e-09$(4)$ \\

3 & 5.23e-13$(3)$ & 5.06e-13$(3)$ & 5.03e-13$(3)$ \\
\hline
\end{tabular}
}
\end{table}

\begin{figure}[!htbp]
\centering
\captionsetup{justification=centering}
  \subfloat[$\alpha = 0.25$]{
  \includegraphics[width=0.32\textwidth]{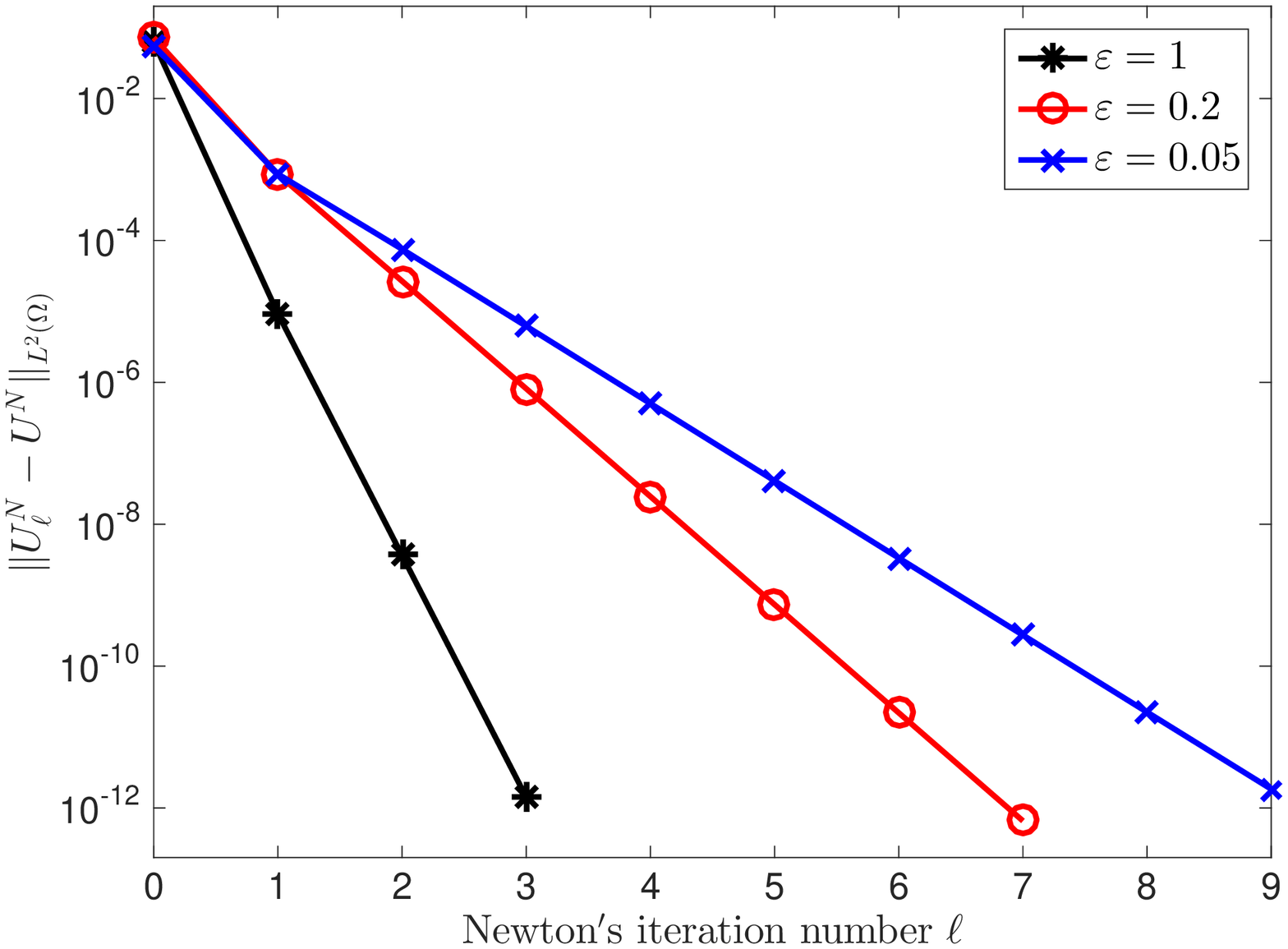}
  \label{fig:eg4-alpha1}
} %
  \subfloat[$\alpha = 0.75$]{
  \includegraphics[width=0.32\textwidth]{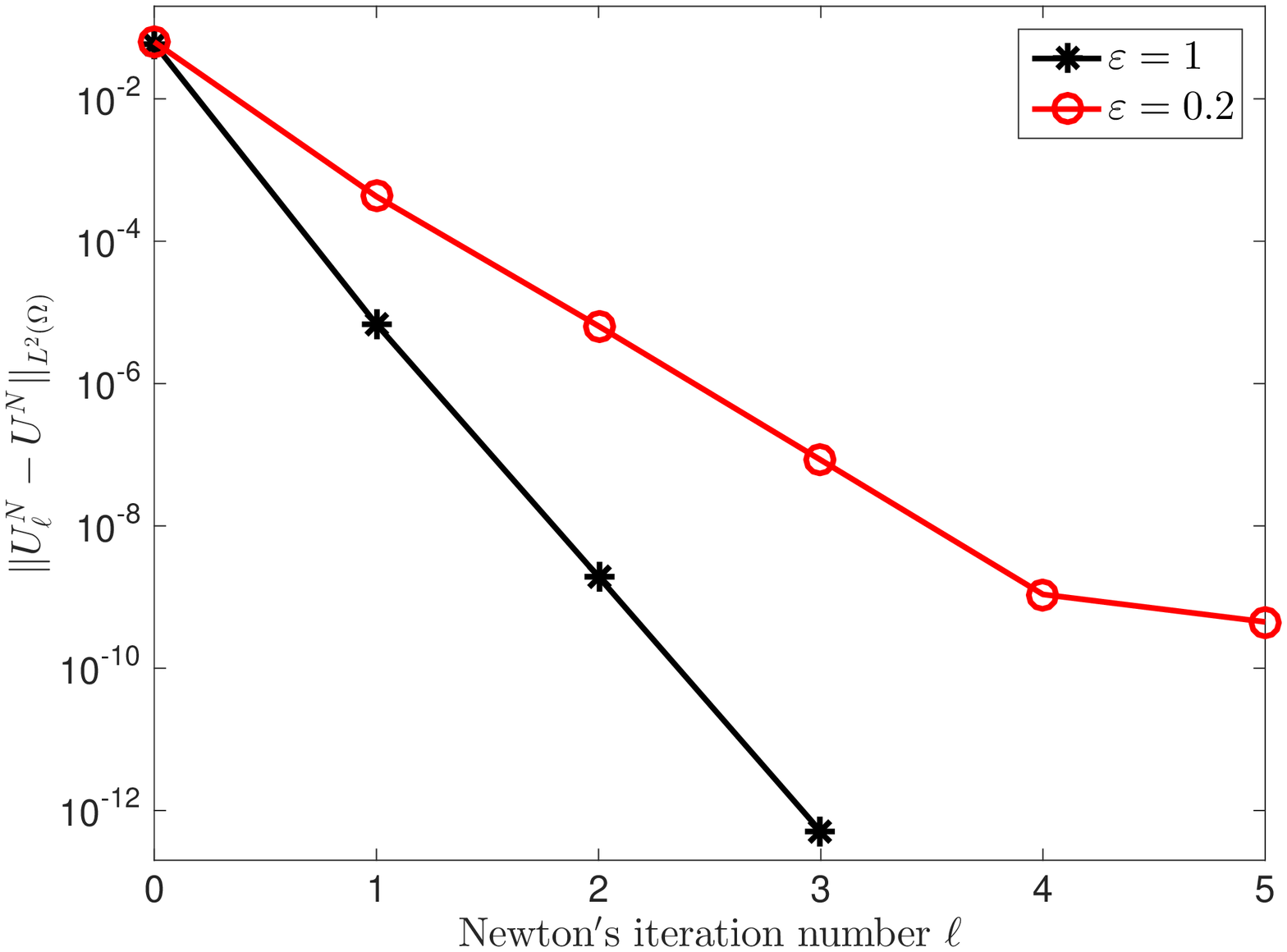}
  \label{fig:eg4-alpha2}
} %
  \subfloat[$\alpha = 1$]{
  \includegraphics[width=0.32\textwidth]{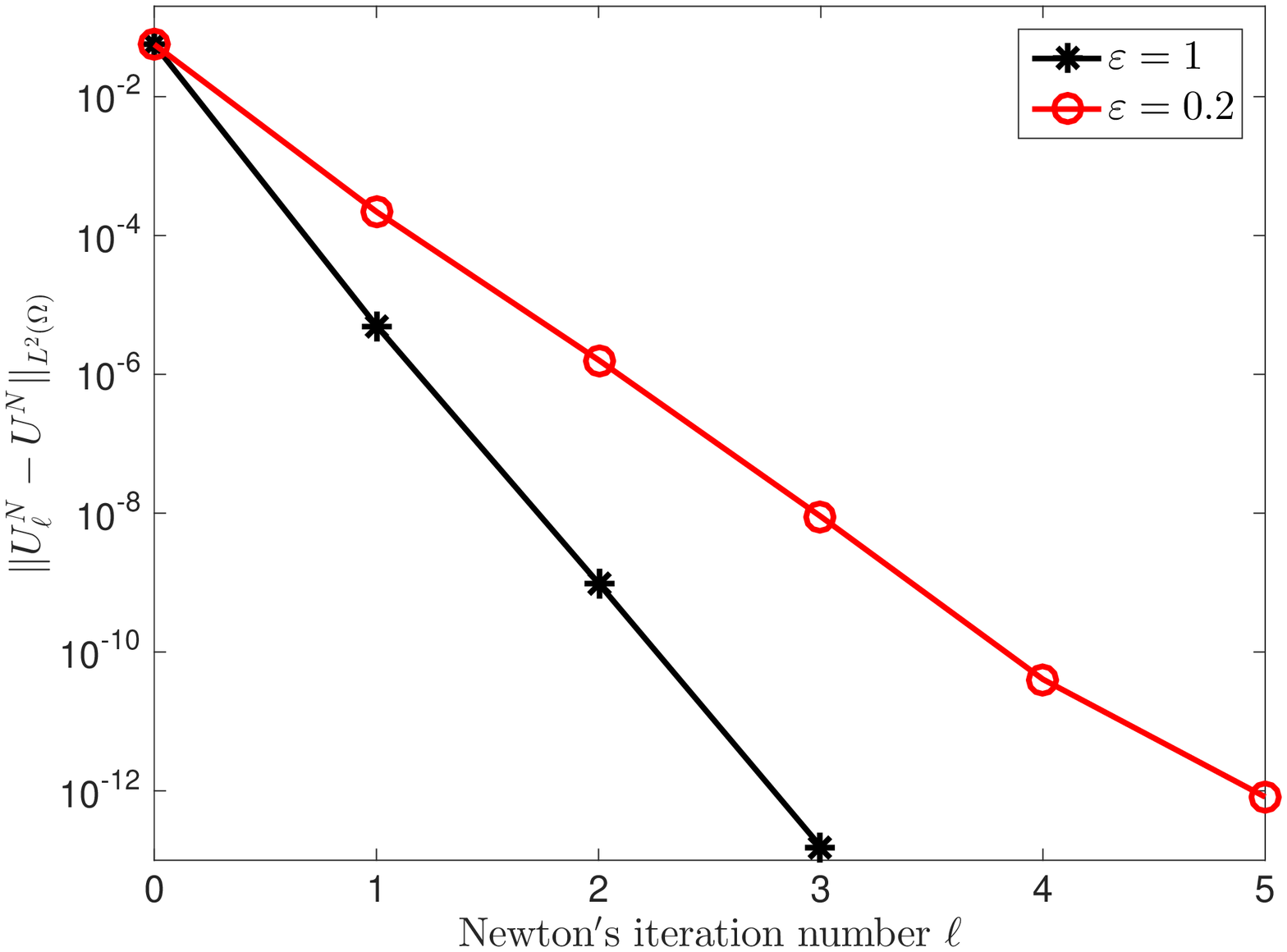}
  \label{fig:eg4-alpha3}
}
\caption{PinT BDF3 for Example \ref{Example4}: Influence of the strength of nonlinearity.}
\label{fig:eg4-alpha}
\end{figure}

Next, we investigate the influence of the strength of nonlinearity for both subdiffusion and normal diffusion cases in Figure \ref{fig:eg4-alpha}.
As can be seen from  Table \ref{tab:fde-nonlinear}, the convergence behaviors of Algorithm \ref{alg:nonlinear} are insensitive with various BDF$k$ schemes, and hence the results of BDF2 scheme are presented. We observe that strong nonlinearity will lower the convergence rate, not only due to the strong nonlinearity itself to the Newton's iteration, but also possibly to the more inaccurate average $\overline U_{\ell-1}$ in \eqref{eqn:BDF-CQ-m-r-N}. As the nonlinearity getting stronger, for example $\varepsilon = 0.05$ with $\alpha = 0.75$ or $\alpha = 1$, the modified Newton's iteration does not converge. It is reasonable
that the accuracy of average $\overline U_{\ell-1}$ hinges on the variation of the solutions on certain time interval.
Practically, a windowing technique could be used in this algorithm: after a certain number
of time steps computed in parallel in the current time window, the computation can be restarted for the
next time window in a sequential way. This is beyond the scope of current paper and can be considered in the future.



}
\end{example}


\section*{Acknowledgements}
The research of S. Wu is partially supported by the National Natural
Science Foundation of China grant (No. 11901016) and the startup grant
from Peking University.  The research of Z. Zhou is partially
supported by a Hong Kong RGC grant (project No. 15304420).

\bibliographystyle{abbrv}
\bibliography{ref_parallel}
\end{document}